\documentclass[a4paper,11pt]{article}
\usepackage[T1]{fontenc}
\usepackage[utf8]{inputenc}
\usepackage{amsmath}
\usepackage{amsthm}
\usepackage{amssymb}
\usepackage{amsfonts}
\usepackage{graphics}
\usepackage{color}
\newcounter{theorem}

\newtheorem{theo}[theorem]{Theorem}%
\newtheorem{defi}[theorem]{Definition}%
\newtheorem{lemm}[theorem]{Lemma}%
\newtheorem{coro}[theorem]{Corollary}%
\newtheorem{rema}[theorem]{Remark}%

\def\eef{{\bf E}}
\def\rrf{{\bf R}}

\def\ec{{\cal E}}
\def\mc{{\cal M}}
\def\rc{{\cal R}}

\def\one{{\bf 1}}
\def\d{{\rm d}}

\def\fdc{]\hskip-1mm]}
\def\odc{[\hskip-1mm[}

\usepackage{arydshln,leftidx,mathtools}
\def \ep{\varepsilon}

\def \diag{{\rm Diag}}
\def \diam{{\rm diam}}
\def \supp{{\rm Supp}}
\newcommand{\argmin}{\mathop{\arg\,\min}\limits}

\begin{document}
\title{Iterated proportional fitting procedure
and infinite products of stochastic matrices}
\author{Jean Brossard and Christophe Leuridan}
\date{\today}
\maketitle

\begin{abstract}
The iterative proportional fitting procedure (IPFP), introduced 
in 1937 by Krui\-thof, aims to adjust the elements of an array 
to satisfy specified row and column sums. Thus, given a 
rectangular non-negative matrix $X_0$ and two positive marginals 
$a$ and $b$, the algorithm generates a sequence of matrices 
$(X_n)_{n \ge 0}$ starting at $X_0$, supposed to converge to a 
biproportional fitting, that is, to a matrix $Y$ whose marginals 
are $a$ and $b$ and of the form $Y=D_1X_0D_2$, for some diagonal 
matrices $D_1$ and $D_2$ with positive diagonal entries. 
 
When a biproportional fitting does exist, it is unique and the  
sequence $(X_n)_{n \ge 0}$ converges to it at an at least geometric rate. 
More generally, when there exists some matrix with marginal $a$ and $b$ 
and with support included in the support of $X_0$, the  
sequence $(X_n)_{n \ge 0}$ converges to the unique matrix whose marginals 
are $a$ and $b$ and which can be written as a limit of matrices 
of the form $D_1X_0D_2$. 

In the opposite case (when there exists no matrix with marginals 
$a$ and $b$ whose support is included in the support of $X_0$), 
the sequence $(X_n)_{n \ge 0}$ diverges but both subsequences 
$(X_{2n})_{n \ge 0}$ and $(X_{2n+1})_{n \ge 0}$ converge. 

In the present paper, we use a new method to prove again these 
results and determine the two limit-points in the case of 
divergence. Our proof relies on a new convergence theorem for backward infinite 
products $\cdots M_2M_1$ of stochatic matrices $M_n$, with diagonal entries 
$M_n(i,i)$ bounded away from $0$ and with bounded ratios 
$M_n(j,i)/M_n(i,j)$. This theorem generalizes Lorenz' stabilization theorem. 

We also provide an alternative proof of Touric and Nedi\'c's theorem 
on backward infinite products of doubly-stochatic matrices, 
with diagonal entries bounded away from~$0$. 
In both situations, we improve slightly the conclusion, 
since we establish not only the convergence of the sequence 
$(M_n \cdots M_1)_{n \ge 0}$, but also its finite variation. 
\end{abstract}

\noindent
Keywords: infinite products of stochastic matrices - 
contingency matrices - distri\-butions with given marginals - 
iterative proportional fitting - relative entropy - 
I-divergence. \\
MSC Classification: 15B51 
- 62H17 
- 62B10 
- 68W40.

\section{Introduction}

\subsection{The iterative proportional fitting procedure}

The aim of the iterative proportional fitting procedure is to find
a non-negative matrix with given row and columns sums and having zero
entries at some given places.  
Fix two integers $p \ge 2$, $q \ge 2$ (namely the sizes of the 
matrices to be considered) and two vectors 
$a=(a_1,\ldots,a_p)$, $b=(b_1,\ldots,b_q)$ with positive components
such that $a_1+\cdots+a_p = b_1+\cdots+b_q = 1$ (namely the target 
marginals). We assume that the common value of the sums 
$a_1+\cdots+a_p$ and $b_1+\cdots+b_q$ is $1$ for convenience only, 
to enable probabilistic interpretations, but this is not a true 
restriction.

We introduce the following notations for any $p \times q$ real matric $X$: 
$$X(i,+) = \sum_{j=1}^q X(i,j), \quad X(+,j) = \sum_{i=1}^p X(i,j),
\quad X(+,+) = \sum_{i=1}^p\sum_{j=1}^q X(i,j),$$
and we set $R_i(X) = X(i,+)/a_i$, $C_j(X) = X(+,j)/b_j$. 

The IPFP has been introduced in 1937 by Kruithof~\cite{Kruithof} 
(third appendix) to 
estimate telephone traffic between central stations. 
This procedure starts from a $p \times q$ non-negative matrix $X_0$ 
such that the sum of the entries on each row or column is positive 
(so $X_0$ has at least one positive entry on each row or column) 
and works as follows. 
\begin{itemize}
\item For each $i \in \odc 1,p \fdc$, divide the row $i$ of $X_0$ by  
the positive number $R_i(X_0)$. 
This yields a matrix $X_1$ satisfying the same assumptions as $X_0$ 
and having the desired row-marginals. 
\item For each $j \in \odc 1,q \fdc$, divide the column $j$ of $X_1$ by  
the positive number $C_j(X_1)$. 
This yields a matrix $X_2$ satisfying the same assumptions as $X_0$ 
and having the desired column-marginals. 
\item Repeat the operations above starting from $X_2$ to get $X_3$, $X_4$, 
and so on. 
\end{itemize}
Denote by $\mc_{p,q}(\rrf_+)$ the set of all $p \times q$ matrices with 
non-negative entries, and consider the following subsets: 
\begin{eqnarray*}
\Gamma_0 &:=& \{X \in \mc_{p,q}(\rrf_+) : 
\forall i \in \odc 1,p \fdc,~ X(i,+) > 0,~
\forall j \in \odc 1,q \fdc,~X(+,j) > 0 \}, \\
\Gamma_1 &:=& \{X \in \Gamma_0 : X(+,+) = 1 \} \\
\Gamma_R  &:=& \Gamma(a,*) 
= \{X \in \Gamma_0 : \forall i \in \odc 1,p \fdc,~X(i,+) = a_i \}, \\ 
\Gamma_C &:=& \Gamma(*,b) 
= \{X \in \Gamma_0 : \forall j \in \odc 1,q \fdc,~X(+,j) = b_j \},\\
\Gamma &:=& \Gamma(a,b) = \Gamma_R \cap \Gamma_C. 
\end{eqnarray*}
For every integer $m \ge 1$, denote by $\Delta_m$ the set of all $m \times m$ 
diagonal matrices with positive diagonal entries. 

The IPFP consists in applying 
alternatively the transformations $T_R : \Gamma_0 \to \Gamma_R$ and 
$T_C : \Gamma_0 \to \Gamma_C$ defined by 
$$T_R(X)(i,j) = R_i(X)^{-1}X(i,j) \text { and } 
T_C(X)(i,j) = C_j(X)^{-1}X(i,j).$$ 
The homogeneity of the map $T_R$ shows that replacing $X_0$ 
with $X_0(+,+)^{-1}X_0$ does not change the matrices $X_n$ 
for $n \ge 1$, so there is no restriction to assume that $X_0 \in \Gamma_1$.  

Note that $\Gamma_R$ and $\Gamma_C$ are subsets of $\Gamma_1$ 
and are closed subsets of $\mc_{p,q}(\rrf_+)$. Therefore, if 
$(X_n)_{n \ge 0}$ converges, its limit belongs to the set $\Gamma$. 
Furthermore, by construction, the matrices 
$X_n$ belong to the set 
\begin{eqnarray*}\Delta_pX_0\Delta_q
&=& \{D_1X_0D_2 : D_1 \in \Delta_p, D_2  \in \Delta_q\} \\
&=& \{(\alpha_iX_0(i,j)\beta_j) : 
(\alpha_1,\ldots,\alpha_p) \in (\rrf_+^*)^p,
~(\beta_1,\ldots,\beta_q) \in (\rrf_+^*)^q \}.
\end{eqnarray*}
According to the terminology used by Pretzel, we will say that 
$X_0$ and $X_n$ are diagonally equivalent.
In particular, the matrices $X_n$ have by construction the same support, 
where the support of a matrix $X \in \mc_{p,q}(\rrf_+)$ is defined by
$$\supp(X) = \{(i,j) \in \odc 1,p \fdc \times \odc 1,q \fdc : X(i,j) > 0\}.$$ 
Therefore, for every limit point $L$ of $(X_n)_{n \ge 0}$, we have 
$\supp(L) \subset \supp(X_0)$ and this inclusion may be strict. 
In particular, if $(X_n)_{n \ge 0}$ converges, its limit belongs to the set 
$$\Gamma(X_0) := \Gamma(a,b,X_0) 
= \{S \in \Gamma : \supp(S) \subset \supp(X_0)\}.$$ 

When the set $\Gamma(X_0)$ is empty, the sequence $(X_n)_{n \ge 0}$ cannot 
converge, and no precise behavior was established until 2013, when Gietl and 
Reffel showed that both subsequences $(X_{2n})_{n \ge 0}$ and $(X_{2n+1})_{n \ge 0}$
converge~\cite{Gietl - Reffel}.

In the opposite case, namely when $\Gamma(a,b)$ contains 
some matrix with support included in $X_0$, various proofs 
of the convergence of $(X_n)_{n \ge 0}$ are known
(Bacharach~\cite{Bacharach} in 1965, 
Bregman~\cite{Bregman} in 1967, Sinkhorn~\cite{Sinkhorn} in 1967, 
Csisz\'ar~\cite{Csiszar} in 1975, Pretzel~\cite{Pretzel} in 1980 and others 
(see~\cite{Brown - Chase - Pittenger} and \cite{Pukelsheim} to 
get an exhaustive review). Moreover, the limit can be 
described using some probabilistic tools that we introduce now.
 
\subsection{Probabilistic interpretations and tools}

At many places, we shall identify $a$, $b$ and 
matrices $X$ in $\Gamma_1$ with the probability measures on 
$\odc 1,p \fdc$, $\odc 1,q \fdc$ and $\odc 1,p \fdc \times \odc 1,q \fdc$ 
given by $a(\{i\})=a_i$, $b(\{j\})=b_j$ and $X(\{(i,j)\} = X(i,j)$. 
Through this identification, the set $\Gamma_1$ can be seen 
as the set of all probability measures on 
$\odc 1,p \fdc \times \odc 1,q \fdc$ whose marginals charge every 
point; the set $\Gamma$ can be seen as the set of all probability measures
on $\odc 1,p \fdc \times \odc 1,q \fdc$ having marginals $a$ and $b$.
This set is non-empty since it contains the probability $a \otimes b$.

The $I$-divergence, or Kullback-Leibler divergence, also called 
relative entropy, plays a key role in the study of the iterative 
proportional fitting algorithm. For every $X$ and $Y$ in $\Gamma_1$, 
the relative entropy of $Y$ with respect to $X$ is 
$$D(Y||X) = \sum_{(i,j) \in \supp(X)} Y(i,j) \ln\frac{Y(i,j)}{X(i,j)} 
\text{ if } \supp(Y) \subset \supp(X),$$
and $D(Y||X) = +\infty$ otherwise, with the convention $0 \ln 0 = 0$. 
Although $D$ is not a distance, the quantity $D(Y||X)$ measures in 
some sense how much $Y$ is far from $X$ since $D(Y||X) \ge 0$, 
with equality if and only if $Y=X$. 

In 1968, Ireland and Kullback~\cite{Ireland - Kullback} gave an 
incomplete proof of the convergence of $(X_n)_{n \ge 0}$ when $X_0$ 
is positive, relying on the properties of the $I$-divergence.   
Yet, the $I$-divergence can be used to prove the convergence 
when the set $\Gamma(X_0)$ is non-empty, and to determine the limit. 
The maps $T_R$ and $T_C$ can be viewed has $I$-projections on 
$\Gamma_R$ and $\Gamma_C$ in the sense that for every $X \in \Gamma_0$, 
$T_R(X)$ (respectively $T_C(X)$) is the only matrix achieving the least upper 
bound of $D(Y||X)$ over all $Y$ in $\Gamma_R$ (respectively $\Gamma_C$). 

In 1975, Csisz\'ar established (theorem 3.2 in \cite{Csiszar}) that, 
given a finite collection of linear sets $\ec_1,\ldots,\ec_k$ of 
probability distributions on a finite set, and a distribution $R$ such that 
$\ec_1 \cap \ldots \cap \ec_k$ contains some probability distribution 
which is absolutely continuous with regard to $R$, 
the sequence obtained from $R$ 
by applying cyclically the $I$-projections on $\ec_1,\ldots,\ec_k$ 
converges to the $I$-projection of $R$ on $\ec_1 \cap \ldots \cap \ec_k$. 
This result applies to our context (the finite set is 
$\odc 1,p \fdc \times \odc 1,q \fdc$, the linear sets 
$\ec_1$ and $\ec_2$ are $\Gamma_R$ and $\Gamma_C$) and 
shows that if the set $\Gamma = \Gamma_R \cap \Gamma_C$ 
contains some matrix with support included in $X_0$, 
then $(X_n)_{n \ge 0}$ converges to the $I$-projection of 
$X_0$ on $\Gamma$. 

\section{Old and new results}

Since the behavior of the sequence $(X_n)_{n \ge 0}$ depends only on 
the existence or the non-existence of elements of $\Gamma$ with 
support equal to or included in $\supp(X_0)$, we state a criterion which
determines in which case we are. 

Consider two subsets $A$ of $\odc 1,p \fdc$ and $B$ of $\odc 1,q \fdc$ 
such that $X_0$ is null on $A \times B$. 
Note $A^c = \odc 1,p \fdc \setminus A$ and $B^c = \odc 1,q \fdc \setminus B$. 
Then for every $S \in \Gamma(X_0)$, 
$$a(A) = \sum_{i \in A} a_i = \sum_{(i,j) \in A \times B^c} S(i,j) \le 
\sum_{(i,j) \in [1,p] \times B^c} S(i,j) = \sum_{j \in B^c} b_j = b(B^c).$$
If $a(A)=b(B^c)$, $S$ must be null on $A^c \times B^c$. 
If $a(A)>b(B^c)$, we get a contradiction, so $\Gamma(X_0)$ is empty.

Actually, these causes of the non-existence of elements of $\Gamma$ 
with support equal to or included in $\supp(X_0)$ provide necessary and 
sufficient conditions. We give two criteria, the first one was 
already stated by Bacharach~\cite{Bacharach}. Pukelsheim gave 
a different formulation of these conditions in theorems 2 and 3 
of~\cite{Pukelsheim}. We use a different method, 
relying on the theory of linear system of inequalities, 
and give a more precise statement, namely item 3 of 
critical case below. 

\begin{theo}\label{criteria} (Criteria to distinguish the cases)

\begin{enumerate}
\item (Case of incompatibility) 
The set $\Gamma(X_0)$ is empty if and only if there exist two subsets 
$A \subset \odc 1,p \fdc$ and $B \subset \odc 1,q \fdc$ 
such that $X_0$ is null on $A \times B$ and $a(A)>b(B^c)$. 
\item (Critical case) Assume now that $\Gamma(X_0)$ is not empty. Then 
\begin{enumerate}
\item There exists a matrix $S_0 \in \Gamma(X_0)$ whose support contains 
the support of every matrix in $\Gamma(X_0)$. 
\item The support of $S_0$ is strictly contained in $\supp(X_0)$ 
if and only if there exist two non-empty subsets $A$ of $\odc 1,p \fdc$ 
and $B$ of $\odc 1,q \fdc$ such that $X_0$ is null on $A \times B$ and 
$a(A)=b(B^c)$. 
\item More precisely, the support of $S_0$ is the complement in $\supp(X_0)$ 
of the union of all products $A^c \times B^c$ over all non-empty subsets 
$A \times B$ of $\odc 1,p \fdc \times \odc 1,q \fdc$ such that $X_0$ is 
null on $A \times B$ and $a(A)=b(B^c)$. 
\end{enumerate}
\end{enumerate}
\end{theo}

Note that the assumption that $X_0$ has at least a positive entry on each 
row or column prevents $A$ and $B$ from being full when $X_0$ is null 
on $A \times B$. The additional condition $a(A)>b(B^c)$ (respectively  
$a(A)=b(B^c)$) is equivalent to the condition $a(A)+b(B)>1$ 
(respectively $a(A)+b(B)=1$), so rows and column play a symmetric role. 

The condition $a(A)>b(B^c)$ and the positivity of all components of 
$a$ and $b$ also prevent $A$ and $B$ from being empty. 
We will call {\it cause of incompatibility} any (non-empty) 
block $A \times B \subset \odc 1,p \fdc \times \odc 1,q \fdc$ 
such that $X_0$ is null on $A \times B$ and $a(A)>b(B^c)$. 
If the set $\Gamma(X_0)$ is non-empty, we will call 
{\it cause of criticality} any non-empty 
block $A \times B \subset \odc 1,p \fdc \times \odc 1,q \fdc$ 
such that $X_0$ is null on $A \times B$ and $a(A)=b(B^c)$. 
 

Given a convergent sequence $(x_n)_{n \ge 0}$ of vectors 
in some normed vector space $(E,||\cdot||)$, with limit $\ell \in E$, 
we will say that {\it the rate of convergence is geometric} 
(respectively {\it at least geometric}) 
if $0 < \lim_n ||x_n-x_\infty||^{1/n} < 1$
(respectively $\limsup_n ||x_n-x_\infty||^{1/n} < 1$).
  
We now describe the asymptotic behavior of sequence $(X_n)_{n \ge 0}$
in each case. The first case is already well-known. 

\begin{theo}\label{fast convergence}(Case of fast convergence)
Assume that $\Gamma$ contains some matrix with same support as $X_0$. 
Then 
\begin{enumerate}
\item The sequences $(R_i(X_{2n})_{n \ge 0})$ and $(C_j(X_{2n+1})_{n \ge 0})$
converge to $1$ at an at least geometric rate.  
\item The sequence $(X_n)_{n \ge 0}$ converges to some matrix $X_\infty$ 
which has the same support as $X_0$. The rate of convergence 
is at least geometric.  
\item The limit $X_\infty$ is the only matrix in 
$\Gamma\cap \Delta_pX_0\Delta_q$ (in particular $X_0$ and $X_\infty$ 
are diagonally equivalent). 
It is also the unique matrix achieving the minimum
of $D(Y||X_0)$ over all $Y \in \Gamma(X_0)$. 
\end{enumerate}
\end{theo}

For example, if $p=q=2$, $a_1=b_1=2/3$, $a_2=b_2=1/3$, and 
$$X_0 = \frac{1}{4}
\left(\begin{array}{cc}
2 & 1 \\
1 & 0
\end{array}\right).$$ 
Then for every $n \ge 1$, $X_n$ or $X_n^\top$ is equal to
$$\frac{1}{3(2^n-1)}
\left(\begin{array}{cc}
2^n & 2^n-2 \\
2^n-1 & 0
\end{array}\right) 
,$$
depending on whether $n$ is odd or even.
The limit is
$$X_\infty = \frac{1}{3}
\left(\begin{array}{cc}
1 & 1 \\
1 & 0
\end{array}\right).$$ 

The second case is also well-known, except the fact that 
the quantities $R_i(X_{2n})-1$ and $C_j(X_{2n+1})-1$ are $o(n^{-1/2})$. 

\begin{theo}\label{slow convergence}(Case of slow convergence)
Assume that $\Gamma$ contains some matrix with support 
included in $\supp(X_0)$ but contains no matrix with support 
equal to $\supp(X_0)$. Then 
\begin{enumerate}
\item The series 
$$\sum_{n \ge 0} (R_i(X_{2n})-1)^2 \text{ and } 
\sum_{n \ge 0} (C_j(X_{2n+1})-1)^2$$
converge. 
\item The sequences $(\sqrt{n}(R_i(X_{n})-1))_{n \ge 0}$ and 
$(\sqrt{n}(C_j(X_{n})-1))_{n \ge 0}$ converge to $0$. 
In particular, the sequences $(R_i(X_{n}))_{n \ge 0}$ and 
$(C_j(X_{n}))_{n \ge 0}$ converge to $1$. 
\item The sequence $(X_n)_{n \ge 0}$ converges to some matrix 
$X_\infty$ whose support contains the support of every matrix 
in $\Gamma(X_0)$. 
\item The limit $X_\infty$ is the unique matrix achieving the minimum
of $D(Y||X_0)$ over all $Y \in \Gamma(X_0)$. 
\item If $(i,j) \in \supp(X_0) \setminus \supp(X_\infty)$, 
the infinite product  
$R_i(X_0)C_j(X_1)R_i(X_2)C_j(X_3)\cdots$ is infinite. 
\end{enumerate}
\end{theo}

Actually the assumption that $\Gamma$ contains no matrix with 
support equal to $\supp(X_0)$ can be removed; but when this assumption 
fails, theorem~\ref{fast convergence} applies, so 
theorem~\ref{slow convergence} brings nothing new. 
When this assumption holds, the last conclusion of 
theorem~\ref{slow convergence} shows that the rate of convergence 
cannot be in $o(n^{-1-\ep})$ for any $\ep>0$. 

However, a rate of convergence in $\Theta(n^{-1})$ is possible, 
and we do not know whether other rates of slow convergence may occur. 
For example, consider $p=q=2$, $a_1=a_2=b_1=b_2=1/2$, and 
$$X_0 = \frac{1}{3}
\left(\begin{array}{cc}
1 & 1 \\
1 & 0
\end{array}\right).$$ 
Then for every $n \ge 1$, $X_n$ or $X_n^\top$ is equal to
$$\frac{1}{2n+2}
\left(\begin{array}{cc}
1 & n \\
n+1 & 0
\end{array}\right),$$
depending on whether $n$ is odd or even.
The limit is
$$X_\infty = \frac{1}{2}
\left(\begin{array}{cc}
0 & 1 \\
1 & 0
\end{array}\right).$$ 

When $\Gamma$ contains no matrix with support included in $\supp(X_0)$, 
we already know by Gietl and Reffel's theorem~\cite{Gietl - Reffel} 
that both sequences $(X_{2n})_{n \ge 0}$ and $(X_{2n+1})_{n \ge 0}$ 
converge. The convergence may be slow, so Aas gives in~\cite{Aas} an 
algorithm to fasten the convergence. Aas' algorithm finds and exploits 
the block structure associated to the inconsistent problem of finding 
a non-negative matrix whose marginals are $a$ and $b$ and whose support 
is contained in $\supp(X_0)$. The next two theorems give a complete 
description of the two limit points and how to find them. 

\begin{theo}\label{divergence} (Case of divergence)
Assume that $\Gamma$ contains no matrix with support 
included in $\supp(X_0)$. 

Then there exist some positive integer $r \le \min(p,q)$, some 
partitions $\{I_1,\ldots,I_r\}$ of $\odc 1,p \fdc$ 
and $\{J_1,\ldots,J_r\}$ of $\odc 1,q \fdc$ such that:
\begin{enumerate}
\item $(R_i(X_{2n}))_{n \ge 0}$ converges to $\lambda_k = b(J_k)/a(I_k)$ 
whenever $i \in I_k$; 
\item $(C_j(X_{2n+1}))_{n \ge 0}$ converges $\lambda_k^{-1}$ whenever $j \in J_k$; 
\item $X_n(i,j) = 0$ for every $n \ge 0$ whenever 
$i \in I_k$ and $j \in J_{k'}$ with $k < k'$;
\item $X_n(i,j) \to 0$ as $n \to +\infty$ at a geometric rate whenever 
$i \in I_k$ and $j \in J_{k'}$ with $k > k'$;
\item The sequence $(X_{2n})_{n \ge 0}$ converges 
to the unique matrix achieving the minimum
of $D(Y||X_0)$ over all $Y \in \Gamma(a',b,X_0)$,  
where $a'_i/a_i = \lambda_k$ whenever $i \in I_k$;
\item The sequence $(X_{2n+1})_{n \ge 0}$ converges 
to the unique matrix achieving the minimum
of $D(Y||X_0)$ over all $Y \in \Gamma(a,b',X_0)$,
where $b'_j/b_j = \lambda_k^{-1}$ whenever $j \in J_k$; 
\item For every $k \in \odc 1,r \fdc$, $a'(I_k)=b(J_k)$ and $a(I_k)=b'(J_k)$. 
Morevoer, the support of any matrix in $\Gamma(a',b,X_0) \cup \Gamma(a,b',X_0)$
is contained in $I_1 \times J_1 \cup \cdots \cup I_r \times J_r$. 
\item Let $D_1 = \diag(a'_1/a_1,\ldots,a'_p/a_p)$ and 
$D_2 = \diag(b_1/b'_1,\ldots,b_q/b'_q)$. Then for every 
$S \in \Gamma(a,b',X_0)$, $D_1S = SD_2 \in \Gamma(a',b,X_0)$, 
and all matrices in $\Gamma(a',b,X_0)$ can be written in this way.  
\end{enumerate}
\end{theo}

For example, if $p=q=2$, $a_1=b_1=1/3$, $a_2=b_2=2/3$, and 
$$X_0 = \frac{1}{3}
\left(\begin{array}{cc}
1 & 1 \\
1 & 0
\end{array}\right).$$ 
Then for every $n \ge 1$, $X_n$ or $X_n^\top$ is equal to
$$\frac{1}{3(3 \times 2^{n-1}-1)}
\left(\begin{array}{cc}
1 & 3 \times 2^{n-1}-2 \\
2(3 \times 2^{n-1}-1) & 0
\end{array}\right), 
$$
depending on whether $n$ is odd or even. 
We get $a'_1=b'_1=2/3$ and $a'_2=b'_2=1/3$ since
the two limit points are
$$\frac{1}{3}
\left(\begin{array}{cc}
0 & 1 \\
2 & 0
\end{array}\right)
\text{ and }~ \frac{1}{3}
\left(\begin{array}{cc}
0 & 2 \\
1 & 0
\end{array}\right).$$ 

The symmetry in theorem~\ref{divergence} shows that the limit 
points of $(X_n)_{n \ge 0}$ would be the same if we would applying 
$T_C$ first instead of $T_R$. 

Actually, the assumption that $\Gamma(X_0)$ is empty can be removed 
and it is not used in the proof of theorem~\ref{divergence}. Indeed, 
when $\Gamma(X_0)$ is non-empty, the conclusions still hold with $r=1$ 
and $\lambda_1=1$, but theorem~\ref{divergence} brings nothing new 
in this case. 

Theorem~\ref{divergence} does not indicate what the partitions 
$\{I_1,\ldots,I_r\}$ and $\{J_1,\ldots,J_r\}$ are.   
Actually the integer $r$, the partitions 
$\{I_1,\ldots,I_r\}$ and $\{J_1,\ldots,J_r\}$ depend 
only on $a$, $b$ and on the support of $X_0$. 
and can be determined recursively as follows. 
This gives a complete description 
of the two limit points mentioned in theorem~\ref{divergence}.

\begin{theo}\label{determining the partitions}
(Determining the partitions in case of divergence)
Keep the assumption and the notations of theorem~\ref{divergence}. 
Fix $k \in \odc 1,r \fdc$, set 
$P = \odc 1,p \fdc \setminus (I_1 \cup \ldots \cup I_{k-1})$, 
$Q = \odc 1,q \fdc \setminus (J_1 \cup \ldots \cup J_{k-1})$, 
and consider the restricted problem associated to the marginals 
$a(\cdot|P) = (a_i/a(P))_{i \in P}$, $b(\cdot|Q) = (b_j/b(Q))_{j \in Q}$
and to the initial condition $(X_0(i,j))_{(i,j) \in P \times Q}$. 

If $k=r$, this restricted problem admits some solution.  

If $k<r$, the set 
$A_k \times B_k := I_k \times (Q \setminus J_k)$ is a 
cause of incompatibility of this restricted problem. 
More precisely, among all causes of incompatibility 
$A \times B$ maximizing the ratio $a(A)/b(Q \setminus B)$, it 
is the one which maximizes the set $A$ and minimizes 
the set $B$.  
\end{theo}

Note that if a cause of incompatibility $A \times B$ 
maximizes the ratio $a(A)/b(B^c)$, then it is
maximal for the inclusion order.
We now give an example to illustrate how 
theorem~\ref{determining the partitions} enables us to
determine the partitions $\{I_1,\ldots,I_r\}$ and $\{J_1,\ldots,J_r\}$.
In the following array, the $*$ and $0$ indicate the positive and the 
null entries in $X_0$; the last column and row indicate the target 
sums on each row and column. 
$$\begin{array}{cccccc|c}
* & * & 0 & 0 & 0 & 0 & .25 \\
0 & * & * & 0 & 0 & 0 & .25 \\
0 & * & * & * & 0 & 0 & .25 \\
* & * & * & * & 0 & * & .15 \\
* & 0 & * & * & * & * & .10 \\
\hline
.05 &.05 &.1 &.2 &.2 &.4 & 1
\end{array}$$
We indicate below in underlined boldface characters 
some blocks $A \times B$ of zeroes which 
are causes of incompatibility, and the 
corresponding ratios $a(A)/b(B^c)$. 
$$\begin{array}{l}
A = \{1,2,3,4\} \\
B = \{5\} \\
\displaystyle{\frac{a(A)}{b(B^c)} = \frac{0.9}{0.8}}
\end{array}
\quad
\begin{array}{cccccc|c}
* & * & 0 & 0 & \underline{\bf 0} & 0 & .25 \\
0 & * & * & 0 & \underline{\bf 0} & 0 & .25 \\
0 & * & * & * & \underline{\bf 0} & 0 & .25 \\
* & * & * & * & \underline{\bf 0} & * & .15 \\
* & 0 & * & * & * & * & .10 \\
\hline
.05 &.05 &.1 &.2 &.2 &.4 & 1
\end{array}$$

$$\begin{array}{l}
A = \{2,3\} \\
B = \{1,5,6\} \\
\displaystyle{\frac{a(A)}{b(B^c)} = \frac{0.5}{0.35}}
\end{array}
\quad
\begin{array}{cccccc|c}
* & * & 0 & 0 & 0 & 0 & .25 \\
\underline{\bf 0} & * & * & 0 & \underline{\bf 0} & \underline{\bf 0} & .25 \\
\underline{\bf 0} & * & * & * & \underline{\bf 0} & \underline{\bf 0} & .25 \\
* & * & * & * & 0 & * & .15 \\
* & 0 & * & * & * & * & .10 \\
\hline
.05 &.05 &.1 &.2 &.2 &.4 & 1
\end{array}$$

$$\begin{array}{l}
A = \{1\} \\
B = \{3,4,5,6\} \\
\displaystyle{\frac{a(A)}{b(B^c)} = \frac{0.25}{0.1}}
\end{array}
\quad
\begin{array}{cccccc|c}
* & * & \underline{\bf 0} & \underline{\bf 0} & \underline{\bf 0} & \underline{\bf 0} & .25 \\
0 & * & * & 0 & 0 & 0 & .25 \\
0 & * & * & * & 0 & 0 & .25 \\
* & * & * & * & 0 & * & .15 \\
* & 0 & * & * & * & * & .10 \\
\hline
.05 &.05 &.1 &.2 &.2 &.4 & 1
\end{array}$$

$$\begin{array}{l}
A = \{1,2\} \\
B = \{4,5,6\} \\
\displaystyle{\frac{a(A)}{b(B^c)} = \frac{0.5}{0.2}}
\end{array}
\quad
\begin{array}{cccccc|c}
* & * & 0 & \underline{\bf 0} & \underline{\bf 0} & \underline{\bf 0} & .25 \\
0 & * & * & \underline{\bf 0} & \underline{\bf 0} & \underline{\bf 0} & .25 \\
0 & * & * & * & 0 & 0 & .25 \\
* & * & * & * & 0 & * & .15 \\
* & 0 & * & * & * & * & .10 \\
\hline
.05 &.05 &.1 &.2 &.2 &.4 & 1
\end{array}$$
One checks that the last two blocks are those which maximize the ratio
$a(A)/b(B^c)$. Among these two blocks, the latter has a bigger $A$ and 
a smaller $B$, so it is $A_1 \times B_1$. 
Therefore, $I_1 = \{1,2\}$ and $J_1 = \{1,2,3\}$, and we look at the 
restricted problem associated to the marginals 
$a(\cdot|I_1^c)$, $b(\cdot|J_1^c)$
and to the initial condition $(X_0(i,j))_{(i,j) \in I_1^c \times J_1^c}$. 
The dots below indicate the removed rows and columns. 
$$\begin{array}{cccccc|c}
\cdot & \cdot & \cdot & \cdot & \cdot & \cdot & \cdot \\
\cdot & \cdot & \cdot & \cdot & \cdot & \cdot & \cdot \\
\cdot & \cdot & \cdot & * & 0 & 0 & .5 \\
\cdot & \cdot & \cdot & * & 0 & * & .3 \\
\cdot & \cdot & \cdot & * & * & * & .2 \\
\hline
\cdot & \cdot & \cdot &.25 &.25 &.5 & 1
\end{array}$$
Two causes of impossibility have to be considered, namely 
$\{3,4\} \times \{5\}$ and $\{3\} \times \{5,6\}$. The latter 
maximizes the ratio $a(A)/b(J_1^c \setminus B)$, 
so it is $A_2 \times B_2$. Therefore, $I_2 = \{3\}$ and 
$J_2 = \{4\}$, and we look at the 
restricted problem below.
$$\begin{array}{cccccc|c}
\cdot & \cdot & \cdot & \cdot & \cdot & \cdot & \cdot \\
\cdot & \cdot & \cdot & \cdot & \cdot & \cdot & \cdot \\
\cdot & \cdot & \cdot & \cdot & \cdot & \cdot & \cdot \\
\cdot & \cdot & \cdot & \cdot & 0 & * & .6 \\
\cdot & \cdot & \cdot & \cdot & * & * & .4 \\
\hline
\cdot & \cdot & \cdot & \cdot &.33 &.67 & 1
\end{array}$$
This time, there is no cause of impossibility, so $r=3$, 
and the sets $I_3 = \{4,5\}$, $J_3 = \{5,6\}$ contain 
all the remaining indices. We shall indicate with dashlines  
the block structure defined by the partitions $\{I_1,I_2,I_3\}$ 
and $\{J_1,J_2,J_3\}$ (for readability, our example was chosen 
in such a way that each block is made of consecutive indices). 
By theorem~\ref{divergence},
the limit of the sequence $(X_{2n})_{n \ge 0}$ admits 
marginals $a'$ and $b$, its support is included in 
$\supp(X_0)$ and also in 
$(I_1\times J_1) \cup (I_2 \times J_2) \cup (I_3 \times J_3)$, 
namely it solves the problem below.
$$\begin{array}{ccc:c:cc|c}
* & * & 0 & 0 & 0 & 0 & .1 \\
0 & * & * & 0 & 0 & 0 & .1 \\
\hdashline
0 & 0 & 0 & * & 0 & 0 & .2 \\
\hdashline
0 & 0 & 0 & 0 & 0 & * & .36 \\
0 & 0 & 0 & 0 & * & * & .24 \\
\hline
.05 &.05 &.1 &.2 &.2 &.4 & 1
\end{array}$$
We observe that each cause of incompatibility $A \times B$ related to the block 
structure, namely $I_1 \times (J_2 \cup J_3)$ or $(I_1 \cup I_2) \times J_3$,
becomes a cause of criticality with regard to the margins $a'$ and $b$, 
namely $a'(A)=b(B^c)$, so $\lim_n X_{2n}$ has zeroes on $A^c \times B^c$.
This statement fails for the ``weaker'' cause of incompatibility 
$\{1,2,3,4\} \times \{5\}$. Yet, it still holds for the 
cause of incompatibility $\{1\} \times \{3,4,5,6\}$; 
this rare situation occurs because there were two causes of 
incompatibility maximizing the ratio $a(A)/b(B^c)$, namely
$\{1\} \times \{3,4,5,6\}$ and $\{1,2\} \times \{4,5,6\}$. 
Hence $\lim_n X_{2n}$ has another additional zero at 
position $(2,2)$. We add dashlines below to make visible this 
refinement of the block structure.   
Here, no minimization of $I$-divergence is required to get 
the limit of $(X_{2n})_{n \ge 0}$ since the set $\Gamma(a',b,X_0)$ contains 
only one matrix, namely
$$\begin{array}{cc:c::c::cc|c}
.05 &.05 & 0 & 0 & 0 & 0 & .1 \\
\hdashline
0 & 0 & 0.1 & 0 & 0 & 0 & .1 \\
\hdashline\hdashline
0 & 0 & 0 & .2 & 0 & 0 & .2 \\
\hdashline\hdashline
0 & 0 & 0 & 0 & 0 & .36 & .36 \\
0 & 0 & 0 & 0 & .2 & .04 & .24 \\
\hline
.05 &.05 &.1 &.2 &.2 &.4 & 1
\end{array}~.$$
We note that the convergence of $X_{2n}(2,2)$ to $0$ is slow since 
$$\lim_{n \to +\infty} \frac{X_{2n+2}(2,2)}{X_{2n}(2,2)} 
= \lim_{n \to +\infty} \frac{1}{R_2(X_{2n})C_2(X_{2n+1})} 
= \frac{1}{\lambda_1\lambda_1^{-1}} = 1.$$
This is a typical situation in which the following observation is useful.

\begin{theo}\label{modification of the initial matrix}
The matrices $\lim_n X_{2n}$ and $\lim_n X_{2n+1}$ have the same 
support $\Sigma$ and $\Sigma$ is the union of the supports 
of all matrices in $\Gamma(a',b,X_0) \cup \Gamma(a,b',X_0)$. 

Morevoer, if $X'_0$ denotes the matrix obtained from $X_0$ by setting to $0$ 
all entries outside $\Sigma$, then the limit points provided by the IPFP 
starting from $X_0$ and from $X'_0$ coincide. 
\end{theo} 

Aas mentions this fact (proposition~1 in~\cite{Aas}) as a result 
of Pretzel (last part of theorem~1 in \cite{Pretzel}), although Pretzel 
considers only the case where the set $\Gamma(a,b,X_0)$ is not empty. 
We prove theorem~\ref{modification of the initial matrix} by adapting
Pretzel's proof.

The set $\Sigma$ can be determined by applying theorem~\ref{criteria} 
(critical case, item~(c)) to the the marginals $a'$ and $b$. The interest of 
theorem~\ref{modification of the initial matrix}
is that starting from $X'_0$ ensures an at least geometric rate of 
convergence, since theorem~\ref{fast convergence}
applies when one performs the IPFP on the marginals $a'$ and $b$ 
(or $a$ and $b'$) and the initial matrix $X'_0$. 

That is why Aas investigates the inherent block structure. Actually, 
the splitting considered by Aas is finer that the splitting provided by our 
theorems~\ref{divergence} and~\ref{determining the partitions}; 
in the example above, Aas would split $I_1$ into $\{1,2\}$ and $\{3\}$, 
and $J_1$ into $\{1\}$ and $\{2\}$.
Up to this distinction, most of the statements provided by our 
theorems~\ref{divergence} and~\ref{determining the partitions}
are explicitely or implicitely present in Aas' paper, 
which focuses on an algorithmic point of view. 

The convergence of the sequences $(X_{2n})_{n \ge 0}$ and $(X_{2n+1})_{n \ge 0}$ 
was already established by Gietl and Reffel~\cite{Gietl - Reffel}
with the help of $I$-divergence. Our proof is completely different 
(although $I$-divergence helps us to determine the limit points). 
Our first step is to prove the convergence 
of the sequences $(R_i(X_{2n}))_{n \ge 0}$ and $(C_j(X_{2n+1}))_{n \ge 0}$ by 
exploiting recursion relations involving stochastic matrices. The 
proof relies on the next general result on infinite products of 
stochastic matrices. Theorem~\ref{infinite product of stochastic matrices}
below will only be used for the proof of theorems~\ref{divergence} 
and~\ref{determining the partitions}, so apart from 
lemma~\ref{decreasing the length}, section~\ref{infinite products} 
can be skipped if the reader is only interested by the new proof of 
theorems~\ref{criteria}, ~\ref{fast convergence} and ~\ref{slow convergence}.

A sequence $(x_n)_{n \ge 0}$ of vectors in some normed vector space 
$(E,||\cdot||)$ will be said to have a {\it finite variation} 
if the series $\sum_n ||x_{n+1}-x_n||$ converges. The finite variation 
implies the convergence when $(E,||\cdot||)$ is a Banach space, 
in particular when $E$ has a finite dimension. 

\begin{theo}\label{infinite product of stochastic matrices}
Let $(M_n)_{n \ge 1}$ be some sequence of $d \times d$ stochastic matrices. 
Assume that there exists some constants $\gamma>0$, and $\rho \ge 1$ 
such that for every $n \ge 1$ and $i,j$ in $\odc 1,d \fdc$, 
$M_n(i,i) \ge \gamma$ and $M_n(i,j) \le \rho M_n(j,i)$. 
Then the sequence $(M_n \cdots M_1)_{n \ge 1}$ has a finite variation, 
so it converges to some stochastic matrix $L$. Moreover, the series 
$\sum_n M_n(i,j)$ and $\sum_n M_n(j,i)$ converge whenever the rows of 
$L$ with indexes $i$ and $j$ are different. 
\end{theo}

An important literature deals with infinite products of stochastic 
matrices, with various motivations: study of inhomogeous Markov chains, 
of opinion dynamics... See for example~\cite{Touri}. 
Backward infinite products converge much more often than forward 
infinite products. Many theorems involve the 
ergodic coefficients of stochastic matrices. For a $d \times d$ 
stochastic matrix $M$, the ergodic coefficient is the quantity
$$\tau(M) = \min_{1 \le i,i' \le d} \sum_{j=1}^d \min(M(i,j),M(i',j)) \in [0,1].$$ 
The difference $1-\tau(M)$ is the maximal total variation distance 
between the lines of $M$ seen as probablities on $\odc 1,d \fdc$. 
These theorems do not apply in our context. 

To our knowledge, theorem~\ref{infinite product of stochastic matrices} 
is new. The closest statements we found in the literature are 
Lorenz' stabilization theorem (theorem~2 of~\cite{Lorenz}) and 
a theorem of Touri and Nedi\'c on infinite product of bistochastic 
matrices (theorem~7 of~\cite{Touri - Nedic}, relying on 
theorem~6 of~\cite{Touri - Nedic (random)}).
 The method we use
to prove theorem~\ref{infinite product of stochastic matrices} is 
different of theirs. 

On the one hand, theorem~\ref{infinite product of stochastic matrices} 
provides a stronger conclusion (namely finite variation and not only 
convergence) and has weaker assumptions than Lorenz' stabilization theorem. 
Indeed, Lorenz assumes that each $M_n$ has a positive diagonal and a symmetric 
support, and that the entries of all matrices $M_n$ are bounded below 
by some $\delta>0$; this entails the assumptions of our 
theorem~\ref{infinite product of stochastic matrices},
with $\gamma = \delta$ and $\rho = \delta^{-1}$. 

On the other hand, Lorenz' stabilization theorem gives more 
precisions on the limit $L = \lim_{n \to +\infty} M_n \ldots M_1$. 
In particular, if the support of $M_n$ does not depends on $n$, 
then Lorenz shows that by applying a same permutation on the 
rows and on the columns of $L$, one gets a block-diagonal matrix 
in which each diagonal block is a consensus matrix, namely a 
stochastic matrix whose rows are all the same.  
This additional conclusion does not hold anymore under our weaker assumptions. 
For example, for every $r \in [-1,1]$, consider the stochastic matrix 
$$M(r) = \frac12
\left(\begin{array}{cc}
1+r & 1-r \\
1-r & 1+r
\end{array}\right).$$
One checks that for every $r_1$ and $r_2$ in $[-1,1]$, 
$M(r_2)M(r_1) = M(r_2r_1)$.
Let $(r_n)_{n \ge 1}$ be any sequence of numbers in $]0,1]$ 
whose infinite product converges to some $\ell>0$. 
Then our assumptions hold with $\gamma = 1/2$ and $\rho = 1$ and 
the matrices $M(r_n)$ have the same support. Yet, the limit of
the products $M(r_n) \cdots M(r_1)$, namely $M(\ell)$, has only 
positive coefficients and is not a consensus matrix. 

Note also that, given an arbitrary sequence $(M_n)_{n \ge 1}$ of 
stochastic matrices, assuming only that the diagonal 
entries are bounded away from $0$ does not ensure the 
convergence of the infinite product $\cdots M_2M_1$. 
Indeed, consider the triangular stochastic matrices 
$$T_0 = 
\left(\begin{array}{ccc}
1 & 0 & 0 \\
0 & 1 & 0 \\
0&1/2&1/2
\end{array}\right),~
T_1 = 
\left(\begin{array}{ccc}
1 & 0 & 0 \\
0 & 1 & 0 \\
1/2&0&1/2
\end{array}\right).$$
A recursion shows that for every $n \ge 1$ and 
$(\ep_1,\ldots,\ep_n) \in \{0,1\}^n$, 
$$T_{\ep_n} \cdots T_{\ep_1} = 
\left(\begin{array}{ccc}
1 & 0 & 0 \\
0 & 1 & 0 \\
r&1-2^{-n}-r&2^{-n}
\end{array}\right), \text{ where } 
r = \sum_{k=1}^n \frac{\ep_k}{2^{n+1-k}}.$$
Hence, one sees that the infinite product 
$\cdots T_0T_1T_0T_1T_0T_1$ diverges.

Yet, for a sequence of {\it doubly-stochastic} matrices, it is 
sufficient to assume that the diagonal entries are bounded away from $0$. 
This result was proved by Touri and Nedi\'c 
(theorem~5 of~\cite{Touri - Nedic (alternative)} or theorem~7 
of~\cite{Touri - Nedic}, relying on 
theorem~6 of~\cite{Touri - Nedic (random)}). We provide a simpler
proof and a slight improvement, showing that the sequence 
$(M_n \ldots M_1)_{n \ge 1}$ not only converges but also has a 
finite variation.

\begin{theo}\label{infinite product of doubly-stochastic matrices}
Let $(M_n)_{n \ge 1}$ be some sequence of $d \times d$ 
doubly-stochastic matrices. 
Assume that there exists some constant $\gamma>0$
such that for every $n \ge 1$ and $i$ in $\odc 1,d \fdc$, 
$M_n(i,i) \ge \gamma$. 
Then the sequence $(M_n \ldots M_1)_{n \ge 1}$ has a finite variation, 
so it converges to some stochastic matrix $L$. Moreover, the series 
$\sum_nM_n(i,j)$ and $\sum_nM_n(j,i)$ converge whenever the rows of 
$L$ with indexes $i$ and $j$ are different. 
\end{theo}

Our proof relies on the following fact: for every column vector 
$V \in \rrf^d$, set 
$$D(V) = \sum_{1 \le i,j \le d}|V(i)-V(j)| \text{ and } 
||V||_1 = \sum_{1 \le i \le d}|V(i)|.$$
We will call {\it dispersion of} $V$ the quantity $D(V)$. 
Then, under the assumptions of 
theorem~\ref{infinite product of doubly-stochastic matrices},
the inequality
$$\gamma||M_{n+1} \cdots M_1V - M_n \cdots M_1V||_1 
\le D(M_n \cdots M_1V) - D(M_{n+1} \cdots M_1V).$$
holds for every $n \ge 0$.


\eject

\section{Infinite products of stochastic matrices}\label{infinite products} 

\subsection{Proof of theorem~\ref{infinite product of stochastic matrices}}

We begin with an elementary lemma. 

\begin{lemm}\label{decreasing the length}
Let $M$ be any $m \times n$ stochastic matrix and $V \in \rrf^n$ 
be a column vector. Denote by $\underline{M}$ the smallest entry of $M$, 
by $\underline{V}$, $\overline{V}$ and 
$\diam(V) = \overline{V} - \underline{V}$ 
the smallest entry, the largest entry and the diameter of $V$. 
Then  
$$\underline{MV}
\ge (1-\underline{M})~\underline{V} + \underline{M}~\overline{V}
\ge \underline{V},$$ 
$$\overline{MV}
\le \underline{M}~\underline{V} + (1-\underline{M})~\overline{V}
\le \overline{V},$$
so 
$$\diam(MV) \le (1-2\underline{M})\diam(V).$$ 
\end{lemm}

\begin{proof}
Call $M(i,j)$ the entries of $M$ and $V(1),\ldots,V(n)$ the entries of $V$.
Let $j_1$ and $j_2$ be indexes such that 
$V(j_1)=\underline{V}$ and $V(j_2)=\overline{V}$. 
Then for every $i \in \odc 1,m \fdc$, 
\begin{eqnarray*}
(MV)_i 
&=& \sum_{j \ne j_2} M(i,j)V(j) + M(i,j_2)V(j_2) \\
&\ge& \sum_{j \ne j_2} M(i,j)\underline{V} + M(i,j_2)\overline{V} \\
&=& \underline{V} + M(i,j_2)(\overline{V}-\underline{V}) \\
&\ge& \underline{V} + \underline{M}~(\overline{V}-\underline{V})\\
&\ge& \underline{V}.
\end{eqnarray*}
The first inequality follows. Applying it to $-V$ yields the second 
inequality. 
\end{proof}

The interesting case is when $n \ge 2$, so $\underline{M} \le 1/2$ 
and $1-2\underline{M} \ge 0$. Yet, the lemma and the proof 
above still apply when $n=1$, since   
$\overline{MV}=\underline{MV}=\overline{V}=\underline{V}$
in this case. 

We now restrict ourselves to square stochastic matrices. 
To every column vector $V \in \rrf^d$, we associate the 
column vector $V^\uparrow \in \rrf^d$ obtained by ordering 
the components in non-decreasing order. In particular 
$V^\uparrow(1) = \underline{V}$ and $V^\uparrow(d) = \overline{V}$.

In the next lemmas and corollary, we establish inequalities 
that will play a key role in the proof of 
theorem~\ref{infinite product of stochastic matrices}.  

\begin{lemm}\label{using the diagonal entries}
Let $M$ be some $d \times d$ stochastic matrix with diagonal entries
bounded below by some constant $\gamma>0$, and $V \in \rrf^d$ 
be a column vector with components in increasing order 
$V(1) \le \ldots \le V(d)$. Let $\sigma$ be a permutation of 
$\odc 1,d \fdc$ such that 
$(MV)(\sigma(1)) \le \ldots \le (MV)(\sigma(d))$.
For every $i \in \odc 1,d \fdc$, set 
$$A_i = \sum_{j=1}^{i-1} M(\sigma(i),j)~[V(i)-V(j)],~
B_i = \sum_{j=i+1}^d M(\sigma(i),j)~[V(j)-V(i)],$$
with the natural conventions $A_1 = B_d = 0$.  
The following statements hold. 
\begin{enumerate}
\item For every $i \in \odc 1,d \fdc$, $(MV)^\uparrow(i)-V^\uparrow(i) = B_i-A_i$.
\item All the terms in the sums defining $A_i$ and $B_i$ are non-negative. 
\item\label{diagonal entry} 
$B_i \ge M(\sigma(i),\sigma(i))~[V(\sigma(i))-V(i)] 
\ge \gamma~[V(\sigma(i))-V(i)]$ whenever $i<\sigma(i)$.
\item\label{orbit} Let $a<b$ in $\odc 1,d \fdc$. 
If the orbit $O(a)$ of $a$ associated to the permutation $\sigma$ 
contains some integer at least equal to $b$, then 
\begin{equation*}
V(b)-V(a) \le \gamma^{-1} \sum_{i \in O(a) \cap \odc 1,b-1 \fdc} B_i\one_{[i < \sigma(i)]}
\le \gamma^{-1} \sum_{i \in O(a) \cap \odc 1,b-1 \fdc} B_i.
\end{equation*}
\item One has
$$\sum_{i=1}^d \big|V(\sigma(i))-V(i)\big| 
\le 2\gamma^{-1}\sum_{i=1}^d B_i \one_{[i < \sigma(i)]}
\le 2\gamma^{-1}\sum_{i=1}^{d-1} B_i.$$
\end{enumerate}
\end{lemm}

\begin{proof}
By assumption, 
$$(MV)^\uparrow(i)-V^\uparrow(i) = 
(MV)(\sigma(i))-V(i) = \sum_{j=1}^d M(\sigma(i),j)~[V(j)-V(i)]
= -A_i+B_i,$$
which yields the first item. The next two items follow directly 
from the assumptions $V(1) \le \ldots \le V(d)$ and 
$M(j,j) \ge \gamma$ for every $j \in \odc 1,d \fdc$.  

Under the assumptions of item~\ref{orbit}, the integer 
$n(a,b) = \min\{n \ge 1 : \sigma^n(a) \ge b\}$ is well-defined and  
\begin{eqnarray*}
V(b)-V(a) 
&\le& V(\sigma^{n(a,b)}(a))-V(a) \\
&\le& \sum_{k=0}^{n(a,b)-1} [V(\sigma^{k+1}(a))-V(\sigma^k(a))]
\one_{\sigma^k(a) < \sigma^{k+1}(a)} \\
&\le& \gamma^{-1} \sum_{k=0}^{n(a,b)-1} B_{\sigma^k(a)} \one_{\sigma^k(a) < \sigma^{k+1}(a)}
\end{eqnarray*}
by item~\ref{diagonal entry}. Item~\ref{orbit} follows. 

Since the sum of $V(\sigma(i))-V(i)$ 
over all $i \in \odc 1,d \fdc$ is null 
and since $V(1) \le \ldots \le V(d)$, one has
$$\sum_{i=1}^d \big|V(\sigma(i))-V(i)\big| = 
2 \sum_{i=1}^d \big(V(\sigma(i))-V(i)\big) \one_{[i < \sigma(i)]}
\le 2\gamma^{-1}\sum_{i=1}^{d-1} B_i \one_{[i < \sigma(i)]},$$
by item~\ref{diagonal entry}. The proof is complete.
%
\end{proof}

We denote by $||\cdot||_1$ the norm on $\rrf^d$ defined as the sum 
of the absolute values of the components.  

\begin{lemm}\label{upper bounds}
Keep the assumptions and the notations of 
lemma~\ref{using the diagonal entries}.
Assume that there exists some constant $\rho \ge 1$ 
such that for every $n \ge 1$ and $i,j$ in $\odc 1,d \fdc$, 
$M(i,j) \le \rho M(j,i)$. Set $C = d(d-1)\max(\gamma^{-1},\rho)$. 
Then the following statements hold
\begin{enumerate}
\item\label{upper bound for $A_i$} 
For every $i \in \odc 1,d \fdc$, 
$A_i \le (i-1)\max(\gamma^{-1},\rho)(B_1+\ldots+B_{i-1})$.
\item\label{upper bound for sums of $B_i$} 
For every $m \in \odc 1,d \fdc$, 
$B_1 + \ldots + B_m \le (1+C+\cdots+C^{m-1}) ||(MV)^\uparrow-V^\uparrow||_1$.
\end{enumerate}
\end{lemm}

\begin{proof}
Given $i \in \odc 2,d \fdc$ and $j \in \odc 1,i-1 \fdc$,  
let us check that 
$$A_{i,j} := M(\sigma(i),j)~[V(i)-V(j)] 
\le \max(\gamma^{-1},\rho)(B_1+\ldots+B_{i-1}).$$
We distinguish two cases. 
\begin{itemize}
\item If the orbit of some $k \in \odc 1,j \fdc$ contains some 
integer at least equal to $i$, then 
inequality~\ref{orbit} of lemma~\ref{using the diagonal entries}
applied with $(a,b) = (k,i)$ yields
$$A_{i,j} \le V(i)-V(j) \le V(i)-V(k) \le 
\gamma^{-1} \sum_{z \in O(k) \cap \odc 1,i-1 \fdc} B_z.$$
\item Otherwise, the orbit of every element of $\odc 1,j \fdc$ 
is contained in $\odc 1,i-1 \fdc$, so the orbit of every element 
of $\odc i,d \fdc$ is contained in $\odc j+1,d \fdc$. 
In particular, the orbits $O(\sigma(i)) = O(i)$ and $O(j)$ are 
disjoint. Applying inequality~\ref{diagonal entry} and 
inequality~\ref{orbit} of lemma~\ref{using the diagonal entries}, 
once with $(a,b) = (\sigma(i),i)$, 
once with $(a,b) = (j,\sigma^{-1}(j))$ yields
\begin{eqnarray*}
A_{i,j} 
&=& M(\sigma(i),j)~[V(i)-V(\sigma(i))]  
+ M(\sigma(i),j)~[V(\sigma(i))-V(\sigma^{-1}(j))] \\
& & \quad + M(\sigma(i),j)~[V(\sigma^{-1}(j)-V(j)] \\
&\le& \one_{\sigma(i)<i}~[V(i)-V(\sigma(i))] 
+ \one_{\sigma^{-1}(j)<\sigma(i)}~\rho M(j,\sigma(i))
[V(\sigma(i))-V(\sigma^{-1}(j))] \\
& & \quad + \one_{j<\sigma^{-1}(j)}~[V(\sigma^{-1}(j))-V(j)] \\
&\le& \gamma^{-1} \sum_{z \in O(i) \cap \odc 1,i-1 \fdc} B_{z}~
+~\rho B_{\sigma^{-1}(j)}~
+~ \gamma^{-1} \sum_{z \in O(j) \cap \odc 1,\sigma^{-1}(j)-1 \fdc} B_{z} \\
&\le& \max(\gamma^{-1},\rho) \sum_{z \in \odc 1,i-1 \fdc} B_{z}.
\end{eqnarray*}
\end{itemize}
In both cases, we have got the inequality
$A_{i,j} \le \max(\gamma^{-1},\rho)(B_1+\ldots+B_{i-1})$. 
Summing over all $j \in \odc 1,i-1 \fdc$ yields 
item~\ref{upper bound for $A_i$}.

Let $m \in \odc 1,d \fdc$. 
Equality $(MV)^\uparrow(i)-V^\uparrow(i) = B_i-A_i$ and 
item~\ref{upper bound for $A_i$} yield
\begin{eqnarray*}
\sum_{i=1}^m B_i 
&\le& \sum_{i=1}^m |(MV)^\uparrow(i)-V^\uparrow(i)| + \sum_{i=1}^m A_i \\ 
&\le& ||(MV)^\uparrow-V^\uparrow||_1
+ \sum_{i=1}^m (i-1)\max(\gamma^{-1},\rho)(B_1+\ldots+B_{i-1}) \\ 
&\le& ||(MV)^\uparrow-V^\uparrow||_1   
+ \sum_{i=1}^m (d-1)\max(\gamma^{-1},\rho)(B_1+\ldots+B_{m-1}) \\ 
&\le& ||(MV)^\uparrow-V^\uparrow||_1 + C\sum_{i=1}^{m-1} B_i
\end{eqnarray*}
The particular case where $m=1$ yields $B_1 \le ||(MV)^\uparrow-V^\uparrow||_1$. 
Item~\ref{upper bound for sums of $B_i$} follows by 
induction.
\end{proof}

\begin{coro}\label{variations}
Let $M$ be some $d \times d$ stochastic matrix with diagonal entries
bounded below by some $\gamma>0$
Assume that there exists some constant $\rho \ge 1$ 
such that for every $n \ge 1$ and $i,j$ in $\odc 1,d \fdc$, 
$M(i,j) \le \rho M(j,i)$. Set $C = d(d-1)\max(\gamma^{-1},\rho)$. 
Then for any column $V \in \rrf^d$ the following statements hold.
\begin{enumerate}
\item\label{increasing linear combination} 
Fix $m \in \odc 1,d \fdc$ and $s \ge 1+(m-1)\max(\gamma^{-1},\rho)$.
$$\sum_{i=1}^m s^{-i}(MV)^\uparrow(i) \ge \sum_{i=1}^m s^{-i}V^\uparrow(i).$$
\item\label{upper bound for $||MV-V||_1$} 
$||MV-V||_1 \le (2+C+\cdots+C^{d-2})||(MV)^\uparrow-V^\uparrow||_1$.
\end{enumerate}
\end{coro}

\begin{proof}
By applying a same permutation to the components of $V$, to the rows 
and to the columns of $M$, one may assume that $V(1) \le \ldots \le V(d)$. 
Let $\sigma$ be a permutation of $\odc 1,d \fdc$ such that 
$(MV)(\sigma(1)) \le \ldots \le (MV)(\sigma(d))$.
Then lemmas~\ref{using the diagonal entries} and~\ref{upper bounds} apply. 

For every $i \in \odc 1,m \fdc$, 
$$(MV)^\uparrow(i) - V^\uparrow(i) = B_i-A_i 
\ge B_i - (m-1)\max(\gamma^{-1},\rho)(B_1+\ldots+B_{i-1}).$$ 
Summing over $i$ yields 
\begin{eqnarray*}
\sum_{i=1}^m s^{-i} ((MV)^\uparrow(i) - V^\uparrow(i))
&\ge& 
\sum_{j=1}^m s^{-j}B_j - 
\sum_{i=2}^m \sum_{j=1}^{i-1} s^{-i}(m-1)\max(\gamma^{-1},\rho) B_j \\
&=& 
\sum_{j=1}^m \Big(s^{-j} - (m-1)\max(\gamma^{-1},\rho) \sum_{i=j+1}^{m} 
s^{-j} \Big) B_j \\
&\ge&
\sum_{j=1}^m \Big(s^{-j} - (m-1)\max(\gamma^{-1},\rho) \frac{s^{-(j+1)}}{1-s^{-1}} \Big) B_j \\
&=& 
\sum_{j=1}^m \frac{s^{-j}}{s-1}~[s-1-(m-1)\max(\gamma^{-1},\rho)]~B_j\\
&\ge& 0.
\end{eqnarray*}
Furthermore, for every $i \in \odc 1,d \fdc$, 
$$\big|(MV)(\sigma(i))-V(\sigma(i))\big| - \big|(MV)(\sigma(i))-V(i)\big| 
\le \big|V(\sigma(i))-V(i)\big|.$$
Summing over $i$ and using the last statements of 
lemma~\ref{using the diagonal entries} and
corollary~\ref{upper bounds} yield
\begin{eqnarray*}
||MV-V||_1 - ||(MV)^\uparrow-V^\uparrow||_1
&\le& \sum_{i=1}^d \big|V(\sigma(i))-V(i)\big| \\
&\le& 2\gamma^{-1} \sum_{i=1}^{d-1} B_i \\
&\le& (1+C+\cdots+C^{d-2})||(MV)^\uparrow-V^\uparrow||_1
\end{eqnarray*}
The proof is complete.
\end{proof}

We now derive the last step of the proof of 
theorem~\ref{infinite product of stochastic matrices}. 
Indeed, applying the next corollary to each vector of the 
canonical basis on $\rrf^d$ yields 
theorem~\ref{infinite product of stochastic matrices}. 

\begin{coro}~\label{finite variation}
Let $(M_n)_{n \ge 1}$ be some sequence of $d \times d$ stochastic matrices. 
Assume that there exists some constants $\gamma>0$, and $\rho \ge 1$ 
such that for every $n \ge 1$ and $i,j$ in $\odc 1,d \fdc$, 
$M_n(i,i) \ge \gamma$ and $M_n(i,j) \le \rho M_n(j,i)$. 
For every column vector $V \in \rrf^d$, 
the sequence of vectors $(V_n)_{n \ge 0} := (M_n \ldots M_1V)_{n \ge 0}$ 
has a finite variation, so it converges. Moreover, the series 
$\sum_n M_n(i,j)$ and $\sum_n M_n(j,i)$ converge whenever the two sequences 
$(V_n(i))_{n \ge 0}$ and $(V_n(j))_{n \ge 0}$ have a different limit.
\end{coro}

\begin{proof}
Fix $s \ge 1+(d-1)\max(\gamma^{-1},\rho)$. For each $n$, 
one can apply corollary~\ref{variations} to the matrix $M_{n+1}$ 
and to the vector $V_n$. 

For every $m \in \odc 1,d \fdc$, the sequence 
$(s^{-1} V_n^\uparrow(1) + \cdots + s^{-m} V_n^\uparrow(m))_{n \ge 0}$
is non-decreasing by corollary~\ref{variations} (first part) and 
bounded above by $s^{-1} V^\uparrow(d) + \cdots + s^{-m} V^\uparrow(d)$, 
thanks to lemma~\ref{decreasing the length}, so it has a finite 
variation and converges.
By difference, each sequence $(V_n^\uparrow(i))_{n \ge 0}$ has 
a finite variation. The convergence of the series 
$\sum_n ||V_{n+1}^\uparrow-V_n^\uparrow||_1$ follows, and 
also $\sum_n ||V_{n+1}-V_n||_1$ by 
corollary~\ref{variations} (second part).

Call $\lambda_1<\ldots<\lambda_r$ the distinct values of  
$\lim_{n \to \infty}V_n(i)$ for $i \in \odc 1,d \fdc$. 
For every $k \in \odc 1,r \fdc$, set 
$I_k = \{i \in \odc 1,d \fdc : \lim_{n \to \infty}V_n(i) = \lambda_k\}$, 
$J_k = I_1 \cup \ldots \cup I_k$ and call $m_k$ the size of $J_k$. 
Fix $\ep>0$ such that 
$2\ep<\gamma\min(\lambda_2-\lambda_1,\ldots,\lambda_r-\lambda_{r-1})$, 
so that the intervals $[\lambda_k-\ep,\lambda_k+\ep]$ are pairwise disjoint. 
Then one can find some non-negative integer $N$, such that 
$V_n(i) \in [\lambda_k-\ep,\lambda_k+\ep]$ for every 
$n \ge N$, $k \in \odc 1,r \fdc$, and $i \in I_k$.  

Given $k \in \odc 1,r-1 \fdc$ and $n \ge N$, we show below that  
$$\sum_{i=1}^{m_k}s^{-i}[V_{n+1}^\uparrow(i)-V_n^\uparrow(i)]
\ge (\lambda_{k+1}-\lambda_k-2\ep)s^{-m_k}
\sum_{i \in J_k} \sum_{j \in J_k^c} M_{n+1}(i,j).$$
This inequality together with the convergence of the sequence 
$(V_n^\uparrow)_{n \ge 0}$ and the inequalities $M_n(j,i) \le \rho M_n(i,j)$ 
will yield the convergence of the series $\sum_n M_n(i,j)$ and 
$\sum_n M_n(j,i)$ for every $(i,j) \in J_k \times J_k^c$. 

Fix $n \ge N$ and a permutation $\sigma$ of $\odc 1,d \fdc$ 
such that $V_{n+1}(\sigma(1)) \le \ldots \le V_{n+1}(\sigma(d))$. 
Note that $\sigma(\odc 1,m_k \fdc) = J_k$. 

The column vector $U_n$ defined by $U_n(j) = \min(V_n(j),\lambda_k+\ep)$ 
has the same $m_k$ least components as $V_n$ 
(corresponding to the indexes $j \in J_k$), 
so $U_n^\uparrow$ have the same $m_k$ first components as $V_n^\uparrow$. 
Furthermore, $V_n(j) - U_n(j) \ge \lambda_{k+1}-\lambda_k-2\ep$ 
for every $j \in J_k^c$. 
Hence, for every $i \in \odc 1,d \fdc$,
\begin{eqnarray*}
V_{n+1}^\uparrow(i)-(M_{n+1}U_n)(\sigma(i)) 
&=& V_{n+1}(\sigma(i))-(M_{n+1}U_n)(\sigma(i)) \\
&=& (M_{n+1}V_n-M_{n+1}U_n))(\sigma(i)) \\
&=& \sum_{j=1}^d M_{n+1}(\sigma(i),j)~(V_n(j) - U_n(j)) \\ 
&\ge& (\lambda_{k+1}-\lambda_k-2\ep) \sum_{j \in J_k^c} M_{n+1}(\sigma(i),j).
\end{eqnarray*}
Fix $s \ge 1+(m_k-1)\max(\gamma^{-1},\rho)$. Then
\begin{eqnarray*}
\sum_{i=1}^{m_k}s^{-i} \Big( V_{n+1}^\uparrow(i)
- (M_{n+1}U_n)(\sigma(i)) \Big)
&\ge& (\lambda_{k+1}-\lambda_k-2\ep)~
\sum_{i=1}^{m_k} s^{-i} \sum_{j \in J_k^c} M_{n+1}(\sigma(i),j) \\ 
&\ge& (\lambda_{k+1}-\lambda_k-2\ep)~
s^{-m_k} \sum_{i=1}^{m_k} \sum_{j \in J_k^c} M_{n+1}(\sigma(i),j) \\. 
&=& (\lambda_{k+1}-\lambda_k-2\ep)~
s^{-m_k} \sum_{i \in J_k} \sum_{j \in J_k^c} M_{n+1}(i,j). 
\end{eqnarray*}
But the rearrangement inequality~\footnote{namely 
$\displaystyle{\sum_{i=1}^d U^\uparrow(d+1-i)V^\uparrow(i) \le 
\sum_{i=1}^d U(i)V(i) \le \sum_{i=1}^d U^\uparrow(i)V^\uparrow(i)}$ 
for every $U$ and $V$ in $\rrf^d$.} and the first part 
of corollary~\ref{variations} yield
\begin{eqnarray*}
\sum_{i=1}^{m_k}s^{-i}(M_{n+1}U_n)(\sigma(i))
&\ge&\sum_{i=1}^{m_k}s^{-i}(M_{n+1}U_n)^\uparrow(i) \\
&\ge& \sum_{i=1}^{m_k}s^{-i}U_n^\uparrow(i) \\
&=& \sum_{i=1}^{m_k}s^{-i}V_n^\uparrow(i).
\end{eqnarray*}
We get the desired inequality by additioning the last two inequalities. 

The proof is complete.
\end{proof}

\subsection{Proof of theorem~\ref{infinite product of doubly-stochastic matrices}}

The proof we give is simpler than the proof of 
theorem~\ref{infinite product of stochastic matrices}, 
although some arguments are very similar. 
We begin with the key lemma. 

\begin{lemm}\label{dispersion}
Let $M$ be some $d \times d$ doubly-stochastic matrix with diagonal 
entries bounded below by some constant $\gamma>0$, and $V \in \rrf^d$ 
be any column vector. Call {\rm dispersion of} $V$ the quantity
$$D(V) = \sum_{1 \le i,j \le d}|V(i)-V(j)|.$$
Then $D(V)-D(MV) \ge \gamma ||MV-V||_1$.  
\end{lemm}

\begin{proof}
On the one hand, for every $i$ and $j$ in $\odc 1,d \fdc$, 
\begin{eqnarray*}
(MV)(i)-(MV)(j) 
&=& \sum_{1 \le k \le d} M(i,k)V(k) - \sum_{1 \le l \le d} M(j,l)V(l) \\
&=& \sum_{1 \le k,l \le d} M(i,k)M(j,l)(V(k)-V(l)), 
\end{eqnarray*}
so 
$$D(MV) = \sum_{1 \le i,j \le d} 
\Big|\sum_{1 \le k,l \le d} M(i,k)M(j,l)(V(k)-V(l))\Big|.$$
On the other hand
\begin{eqnarray*}
D(V) &=& \sum_{1 \le k,l \le d}|V(k)-V(l)| \\
&=& \sum_{1 \le i,j \le d} \sum_{1 \le k,l \le d} M(i,k)M(j,l)|V(k)-V(l)|.
\end{eqnarray*}
By difference, $D(V)-D(MV)$ is the sum over all $i$ and $j$ in $\odc 1,d \fdc$ 
of the non negative quantities
$$\Delta(i,j) = 
\sum_{1 \le k,l \le d} M(i,k)M(j,l)|V(k)-V(l)|
- \Big|\sum_{1 \le k,l \le d} M(i,k)M(j,l)(V(k)-V(l))\Big|.$$
Thus
$$D(V)-D(MV) \ge \sum_{1 \le i \le d} \Delta(i,i).$$
But for every $i \in \odc 1,d \fdc$, 
\begin{eqnarray*}
\Delta(i,i) 
&=& \sum_{1 \le k,l \le d} M(i,k)M(i,l)|V(k)-V(l)| - 0 \\
&\ge& \sum_{1 \le k \le d} M(i,k)M(i,i)|V(k)-V(i)| \\
&\ge& \gamma \sum_{1 \le k \le d} M(i,k)|V(k)-V(i)| \\
&\ge& \gamma~\Big|\sum_{1 \le k \le d} M(i,k)(V(k)-V(i))\Big| \\
&=& \gamma~|(MV)(i)-V(i)|.
\end{eqnarray*}
The result follows.
\end{proof}

We now derive the last step of the proof of 
theorem~\ref{infinite product of doubly-stochastic matrices}. 
Indeed, applying the next corollary to each vector of the 
canonical basis on $\rrf^d$ yields 
theorem~\ref{infinite product of doubly-stochastic matrices}. 

\begin{coro}
Let $(M_n)_{n \ge 1}$ be any sequence of $d \times d$ bistochastic 
matrices with diagonal entries bounded below by some $\gamma>0$.  
For every column vector $V \in \rrf^d$, the sequence 
$(V_n)_{n \ge 0}:=(M_n \ldots M_1V)_{n \ge 0}$ has a finite variation, 
so it converges. Moreover, the series $\sum_n M_n(i,j)$ and 
$\sum_n M_n(j,i)$ converge whenever the two sequences 
$(V_n(i))_{n \ge 0}$ and $(V_n(j))_{n \ge 0}$ have a different limit.
\end{coro}

\begin{proof}
Lemma~\ref{dispersion} yields 
$\gamma||V_{n+1}-V_n||_1 \le D(V_n)-D(V_{n+1})$ for every $n \ge 0$. 
In particular, the sequence $((D(V_n))_{n \ge 0}$ is non-increasing 
and bounded below by $0$, so it converges. The convergence of 
the series $\sum_n||V_{n+1}-V_n||_1$ and the convergence of the 
sequence $(V_n)_{n \ge 0}$ follow. 

Call $\lambda_1<\ldots<\lambda_r$ the distinct values of  
$\lim_{n \to \infty}V_n(i)$ for $i \in \odc 1,d \fdc$. 
For every $k \in \odc 1,r \fdc$, set 
$I_k = \{i \in \odc 1,d \fdc : \lim_{n \to \infty}V_n(i) = \lambda_k\}$, 
and $J_k = I_1 \cup \ldots \cup I_k$. 

The proof of the convergence of the series 
$\sum_n M_{n+1}(i,j)$ for every $(i,j) \in J_k \times J_k^c$. 
works like the proof of corollary~\ref{finite variation}, 
with $r$ replaced by $1$, thanks to lemma~\ref{increasing sums} 
stated below, so the rearrangement inequality becomes an equality.  

Using the equality
$$\sum_{(i,j) \in J_k^c \times J_k} M_{n+1}(i,j) 
= |J_k| - \sum_{(i,j) \in J_k \times J_k} M_{n+1}(i,j)
= \sum_{(i,j) \in J_k \times J_k^c} M_{n+1}(i,j),$$ 
we derive the convergence of the series 
$\sum_n M_{n+1}(i,j)$ for every $(i,j) \in J_k^c \times J_k$. 
The proof is complete.
\end{proof}

\begin{lemm}\label{increasing sums}
Let $M$ be some $d \times d$ doubly-stochastic matrix. 
Then for every column vector $V \in \rrf^d$ and $m \in \odc 1,d \fdc$ 
$$\sum_{i=1}^m (MV)^\uparrow(i) \ge \sum_{i=1}^m V^\uparrow(i).$$ 
\end{lemm}

\begin{proof}
By applying a same permutation to the columns of $M$ and 
to the components of $V$, one may assume that 
$V(1) \le \ldots \le V(d)$. By applying a permutation 
to the rows of $M$, one may assume also that 
$(MV)(1) \le \ldots \le (MV)(d)$. 
Since $M$ is doubly-stochastic, the real numbers 
$$S(j) = \sum_{i=1}^m M(i,j) \text{ for } j \in \odc 1,d \fdc$$
are in $[0,1]$ and add up to $m$. 
Morevoer
$$\sum_{i=1}^m (MV)(i) = \sum_{j=1}^d S(j)V(j).$$ 
Hence 
\begin{eqnarray*}
\sum_{i=1}^m (MV)(i) - \sum_{j=1}^m V(j)
&=& \sum_{j=m+1}^d S(j)V(j) + \sum_{j=1}^m (S(j)-1)V(j) \\
&\ge& \sum_{j=m+1}^d S(j)V(m) + \sum_{j=1}^m (S(j)-1)V(m) \\
&=& 0. 
\end{eqnarray*}
We are done. 
\end{proof}

\section{Proof of theorem~\ref{criteria}}

\subsection{Condition for the non-existence of a solution with support included in $\supp(X_0)$}\label{incompatibility}

We assume that $\Gamma$ contains no matrix with support 
included in $\supp(X_0)$, namely that the system
$$\left\{
\begin{array}{r}
\forall i \in \odc 1,p \fdc,~X(i,+) = a_i \\ 
\forall j \in \odc 1,q \fdc,~X(+,j) = b_j \\
\forall (i,j) \in \odc 1,p \fdc \times \odc 1,q \fdc,~X(i,j) \ge 0 \\
\forall (i,j) \in \supp(X_0)^c,~X(i,j) = 0
\end{array}
\right.$$
is inconsistent. 

This system can be seen as a system of linear inequalities of the form 
$\ell(X) \le c$ (where $\ell$ is some linear form and $c$ some constant)
by splitting each equality $\ell(X) = c$ into the two inequalities 
$\ell(X) \le c$ and $\ell(X) \ge c$, and by transforming each inequality 
$\ell(X) \ge c$ into the equivalent inequality $-\ell(X) \le -c$. 
But theorem~4.2.3 in \cite{Webster} (a consequence of Farkas' or 
Fourier's lemma) states that a system of linear inequalities of 
the form $\ell(X) \le c$ is inconsistent 
if and only if some linear combination with non-negative weights of 
the linear inequalities yields the inequality $0 \le -1$.

Consider such a linear combination and call $\alpha_{i,+}$, $\alpha_{i,-}$
$\beta_{j,+}$, $\beta_{j,-}$, $\gamma_{i,j,+}$, $\gamma_{i,j,-}$ the weights 
associated to the inequalities
$X(i,+) \le a_i$, $-X(i,+) \le -a_i$, $X(+,j) \le b_j$, $-X(+,j) \le -b_j$,
$X(i,j) \le 0$, $-X(i,j) \le 0$. When $(i,j) \in \supp(X_0)$, the inequality 
$X(i,j) \le 0$ does not appear in the system, so we set $\gamma_{i,j,+}=0$. 
Then the real numbers 
$\alpha_i := \alpha_{i,+}-\alpha_{i,-}$, 
$\beta_j := \beta_{j,+}-\beta_{j,-}$,
$\gamma_{i,j} := \gamma_{i,j,+} - \gamma_{i,j,-}$
satisfy the following conditions:
\begin{itemize}
\item for every $(i,j) \in \odc 1,p \fdc \times \odc 1,q \fdc$, 
$\alpha_i + \beta_j + \gamma_{i,j} = 0$, 
\item $\displaystyle{\sum_{i=1}^p \alpha_ia_i + \sum_{j=1}^q \beta_jb_j = -1,}$
\item $\gamma_{i,j} \le 0$ whenever $(i,j) \in \supp(X_0)$.
\end{itemize}
Let $U$ and $V$ be two random variables with respective laws  
$$\sum_{i=1}^p a_i\delta_{\alpha_i} \text{ and } \sum_{j=1}^q b_j\delta_{-\beta_j}.$$
Then 
$$\int_\rrf \big(P[U > t] - P[V > t]\big) \d t
= \eef[U]-\eef[V] = \sum_{i=1}^p \alpha_ia_i + \sum_{j=1}^q \beta_jb_j = -1 < 0,$$ 
so there exists some real number $t$ such that $P[U > t] - P[V > t] < 0$.
Consider the sets $A = \{i \in \odc 1,p \fdc : \alpha_i \le t\}$ and
$B = \{j \in \odc 1,q \fdc : \beta_j < -t\}$. Then for every 
$(i,j) \in A \times B$, $-\gamma_{i,j} = \alpha_i + \beta_j < 0$, so 
$(i,j) \notin \supp(X_0)$. In other words, $X_0$ is null on $A \times B$. 
Moreover,
$$a(A)-b(B^c) 
= \sum_{i \in A} a_i - \sum_{j \in B^c} b_j 
= P[U \le t] - P[-V \ge -t] = P[V>t]-P[U>t] > 0.$$
Hence the block $A \times B$ is a cause of incompatibility. The proof is 
complete.

\subsection{Condition for the existence of additional zeroes shared by every solution in $\Gamma(X_0)$}\label{additional zeroes}

We now assume that $\Gamma$ contains some matrix with support 
included in $\supp(X_0)$.

Using the convexity of $\Gamma(X_0)$, one can construct a matrix 
$S_0 \in \Gamma(X_0)$ whose support contains the support of every 
matrix in $\Gamma(X_0)$. This yields item~(a). 

We now prove items~(b) and~(c). The observations made in the introduction 
show that $\supp(S_0) \subset \supp(X_0)$ and that $S_0$ is null on 
$A^c \times B^c$ whenever $A \times B$ is a non-empty subset of 
$\odc 1,p \fdc \times \odc 1,q \fdc$ such that $X_0$ is 
null on $A \times B$ and $a(A)=b(B^c)$. This yields the `if' part of item~(b)
and one inclusion in item~(c). 

To prove the `only if' part of item~(b) and the reverse inclusion in item~(c), 
fix $(i_0,j_0) \in \supp(X_0) \setminus \supp(S_0)$.
Then for every $p \times q$ matrix $X$ with real entries, 
$$\left.
\begin{array}{r}
\forall i \in \odc 1,p \fdc,~X(i,+)=a_i \\ 
\forall j \in \odc 1,q \fdc,~X(+,j)=b_j \\
\forall (i,j) \in \odc 1,p \fdc \times \odc 1,q \fdc,~X(i,j) \ge 0 \\
\forall (i,j) \in \supp(X_0)^c,~X(i,j) = 0
\end{array}
\right\} \implies X(i_0,j_0) \le 0.$$
The system in the left-hand side of the implication is consistent since
$\Gamma(X_0)$ is non-empty by assumption. 
As before, the system at the left-hand side of the implication can be 
seen as a system of linear inequalities of the form $\ell(X) \le c$. 

We now use theorem~4.2.7 in~\cite{Webster} 
(a consequence of Farkas' or Fourier's lemma)
which states that any linear inequation which is a consequence 
of some consistent system of linear inequalities of the form 
$\ell(X) \le c$ can be deduced from the system and from the 
inequality $0 \le 1$ by linear combinations with non-negative weights. 

Consider such a linear combination and call $\alpha_{i,+}$, $\alpha_{i,-}$
$\beta_{j,+}$, $\beta_{j,-}$, $\gamma_{i,j,+}$, $\gamma_{i,j,-}$ and $\eta$ 
the weights associated to the inequalities
$X(i,+) \le a_i$, $-X(i,+) \le -a_i$, $X(+,j) \le b_j$, $-X(+,j) \le -b_j$,
$X(i,j) \le 0$, $-X(i,j) \le 0$, and $0 \le 1$. 
When $(i,j) \in \supp(X_0)$, the inequality 
$X(i,j) \le 0$ does not appear in the system, so we set $\gamma_{i,j,+}=0$. 
Then the real numbers 
$\alpha_i := \alpha_{i,+}-\alpha_{i,-}$, 
$\beta_j := \beta_{j,+}-\beta_{j,-}$,
$\gamma_{i,j} := \gamma_{i,j,+} - \gamma_{i,j,-}$
satisfy the following conditions. 
%
\begin{itemize}
\item for every $(i,j) \in \odc 1,p \fdc \times \odc 1,q \fdc$, 
$\alpha_i + \beta_j + \gamma_{i,j} = \delta_{i,i_0}\delta_{j,j_0}$, 
\item $\displaystyle{\sum_{i=1}^p \alpha_ia_i + \sum_{j=1}^q \beta_jb_j 
+ \eta = 0.}$
\item $\gamma_{i,j} \le 0$ whenever $(i,j) \in \supp(X_0)$,
\item $\eta \ge 0$
\end{itemize}
Let $U$ and $V$ be two random variables with respective laws  
$$\sum_{i=1}^p a_i\delta_{\alpha_i} \text{ and } \sum_{j=1}^q b_j\delta_{-\beta_j}.$$
Then 
$$\int_\rrf \big( P[U>t] - P[V>t] \big)~\d t
= \eef[U]-\eef[V] = \sum_{i=1}^p \alpha_ia_i + \sum_{j=1}^q \beta_jb_j 
= -\eta \le 0.$$
Set $u_0 = \min(\alpha_1,\ldots,\alpha_p) = \mathrm{ess}\inf U$ and 
$v_0 = \max(-\beta_1,\ldots,-\beta_q) = \mathrm{ess}\sup V$. 
Then $P[U>t] - P[V>t] \ge 0$ when $t < u_0$ or $t \ge v_0$.
Hence one can find $t \in [u_0,v_0[$ such that $P[U>t] - P[V>t] \le 0$.
Consider the sets $A = \{i \in \odc 1,p \fdc : \alpha_i \le t\}$ and
$B = \{j \in \odc 1,q \fdc : \beta_j < -t\}$. For every 
$(i,j) \in A \times B$, 
$\gamma_{i,j} = \delta_{i,i_0}\delta_{j,j_0} - \alpha_i - \beta_j > 0$, 
so $(i,j) \notin \supp(X_0)$. In other words, $X_0$ is null on $A \times B$. 
Moreover, 
$$a(A) = \sum_{i \in A} a_i = P[U \le t] \ge P[U = u_0] > 0,$$
$$b(B) = \sum_{j \in B} b_j = P[-V < -t] \ge P[V = v_0] > 0,$$
so $A$ and $B$ are non-empty and
$$a(A)-b(B^c) 
= P[U \le t] - P[V \le t] = P[V>t] - P[U>t] \ge 0.$$
The last inequality is necessarily an equality, since otherwise
$A \times B$ would be a cause of incompatibility. Hence $a(A)=b(B^c)$, so 
$A \times B$ is cause of criticality. 
This proves the `only if' part of item~(b).  

We also know that $(i_0,j_0) \notin A \times B$ since $X_0(i_0,j_0)>0$. 
If we had $(i_0,j_0) \in A^c \times B^c$, we would get the 
reverse inclusion in item~(c). Unfortunately, this statement may fail 
with the choice of $t$ made above. 

Assume, that $(i_0,j_0) \notin A^c \times B^c$. Since $a(A)=b(B^c)$, 
the linear system defining $\Gamma(a,b,X_0)$ can be split into  
three independent consistent subsystems, namely 
$$\forall (i,j) \in (A \times B) \cup (A^c \times B^c), X(i,j) = 0$$
and the two systems 
$$\left\{
\begin{array}{l}
\forall i \in I,~X(i,J)=a_i \\ 
\forall j \in J,~X(I,j)=b_j \\
\forall (i,j) \in I \times J,~X(i,j) \ge 0 \\
\forall (i,j) \in (I \times J) \setminus \supp(X_0)^c,~X(i,j) = 0
\end{array}\right.
$$
where the block $I \times J$ is either $A \times B^c$ or $A^c \times B$.

If $(i_0,j_0) \notin A^c \times B^c$, then $(i_0,j_0)$ belongs to one of 
these two blocks, say $I_1 \times J_1$. Since $a(I_1) = b(J_1)$ and 
since the equality $X(i_0,j_0) = 0$ is a consequence of the 
consistent subsystem above with $I \times J = I_1 \times J_1$, one can apply 
the proof of item~(b) to the marginals $a(\cdot|I_1)$ and $b(\cdot|J_1)$ 
and the restriction of $X_0$ on $I_1 \times J_1$. This yields 
a subset $A_1 \times B_1$ of $I_1 \times J_1$ such that $X_0$ is null on 
$A_1 \times B_1$, $a(A_1)=b(B_1^c)$, and $(i_0,j_0) \notin A_1 \times B_1$. 

If $(i_0,j_0) \in A_1^c \times B_1^c$, we are done. Otherwise, $(i_0,j_0)$ 
belongs to one of the two blocks $A_1 \times B_1^c$ or $A_1^c \times B_1$, say  
$I_2 \times J_2$, and the recursive construction goes on. This construction 
necessarily stops after a finite number of steps and produces a set 
$A' \times B'$ such that $X_0$ is null on $A' \times B'$, $a(A')=b(B'^c)$, 
and $(i_0,j_0) \in A'^c \times B'^c$. Item~(c) follows.




\section{Tools and preliminary results}\label{tools}

\subsection{Results on the quantities $R_i(X_{2n})$ and $C_j(X_{2n+1})$}

\begin{lemm}\label{first equalities}
Let $X \in \Gamma_1$. Then
$$\sum_{i=1}^q a_i R_i(X) = \sum_{i,j} X(i,j) = \sum_j b_j = 1$$ 
and for every $j \in \odc 1,q \fdc$, 
$$C_j(T_R(X)) = \sum_{i=1}^p \frac{X(i,j)}{b_j} R_i(X)^{-1}.$$
When $X \in \Gamma_C$, this equality expresses $C_j(T_R(X))$ as 
a weighted (arithmetic) mean of the quantities $R_i(X)^{-1}$, 
with weights $X(i,j)/b_j$.  
\end{lemm}

For every $X \in \Gamma_1$, call $R(X)$ the column vecteur with 
components $R_1(X),\ldots,R_p(X)$ and $C(X)$ the column vecteur with 
components $C_1(X),\ldots,C_q(X)$. Set 
$$\underline{R}(X) = \min_i R_i(X),\ 
\overline{R}(X) = \max_i R_i(X),\
\underline{C}(X) = \min_j C_j (X),\ 
\overline{C}(X) = \max_i C_j(X).$$

\begin{coro}\label{nested intervals}
The intervals 
$$[\overline{C}(X_1)^{-1},\underline{C}(X_1)^{-1}],~ 
[\underline{R}(X_2),\overline{R}(X_2)],~ 
[\overline{C}(X_3)^{-1},\underline{C}(X_3)^{-1}],~ 
[\underline{R}(X_4),\overline{R}(X_4)],\cdots$$
contain $1$ and form a non-increasing sequence.  
\end{coro}

In lemma~\ref{first equalities}, one can invert the roles of the 
lines and the columns. Given $X \in \Gamma_C$, the matrix $T_R(X)$ 
is in $\Gamma_R$ so the quantities $R_i(T_C(T_R(X)))$ can be written as 
weighted (arithmetic) means of the $C_j(T_R(X))^{-1}$. But the 
$C_j(T_R(X))$ can be written as a weighted (arithmetic) means 
of the quantities $R_k(X)^{-1}$. Putting things together, one gets 
weighted arithmetic means of weighted harmonic means. 
Next lemma shows how to transform these into weighted arithmetic 
means by modifying the weights. 

\begin{lemm}\label{stochastic matrix}
Let $X \in \Gamma_C$. Then $R(T_C(T_R(X))) = P(X)R(X)$, where 
$P(X)$ is the $p \times p$ matrix given by
$$P(X)(i,k) = \sum_{j=1}^q \frac{T_R(X)(i,j)T_R(X)(k,j)}{a_ib_jC_j(T_R(X))}.$$
The matrix $P(X)$ is stochastic. Moreover it satisfies 
for every $i$ and $k$ in $\odc 1,p \fdc$, 
$$P(X)(i,i) \ge \frac{\underline{a}}{\overline{b}~\overline{C}(T_R(X))q}$$
and
$$P(X)(k,i) \le \frac{\overline{a}}{\underline{a}} P(X)(i,k).$$
\end{lemm}

\begin{proof}
For every $i \in \odc 1,p \fdc$, 
\begin{eqnarray*}
R_i(T_C(T_R(X))) 
= \sum_{j=1}^q \frac{T_R(X)(i,j)}{a_i} \frac{1}{C_j(T_R(X))}. 
\end{eqnarray*}
But the assumption $X \in \Gamma_C$ yields 
$$1 = \frac{1}{b_j}\sum_{k=1}^p X(k,j) 
= \frac{1}{b_j}\sum_{k=1}^p T_R(X)(k,j) R_k(X).$$
Hence,
\begin{eqnarray*}
R_i(T_C(T_R(X)))
= \sum_{j=1}^q \frac{T_R(X)(i,j)}{a_i} 
\sum_{k=1}^p \frac{T_R(X)(k,j)}{b_jC_j(T_R(X))} R_k(X).
\end{eqnarray*}
These equalities can be written as $R(T_C(T_R(X))) = P(X)R(X)$, 
where $P(X)$ is the $p \times p$ matrix whose entries are given 
in the statement of lemma~\ref{stochastic matrix}. 
By construction, the entries of $P(X)$ are non-negative and 
for every $i \in \odc 1,p \fdc$, 
$$\sum_{k=1}^p P(X)(i,k) 
= \sum_{j=1}^q \frac{T_R(X)(i,j)}{a_ib_jC_j(T_R(X))}\sum_{k=1}^p T_R(X)(k,j)
= \sum_{j=1}^q \frac{T_R(X)(i,j)}{a_i} = 1.$$
Moreover, since
$$\sum_{j=1}^q T_R(X)(i,j)^2 
\ge \frac{1}{q}\Big(\sum_{j=1}^q T_R(X)(i,j)\Big)^2 = \frac{a_i^2}{q}$$
we have 
\begin{eqnarray*}
P(X)(i,i) 
&\ge& \frac{1}{a_i~\overline{b}~\overline{C}(T_R(X))}
\sum_{j=1}^q T_R(X)(i,j)^2 \\
&=& \frac{a_i}{\overline{b}~\overline{C}(T_R(X))~q} \\
&\ge& \frac{\underline{a}}{\overline{b}~\overline{C}(T_R(X))~q}.
\end{eqnarray*}
The last inequality to be proved follows directly from the symmetry 
of the matrix $(a_iP(X)(i,k))_{1 \le i,k \le p}$. 
\end{proof}

\subsection{A function associated to each element of $\Gamma_1$}

\begin{defi}
For every $X$ and $S$ in $\Gamma_1$, we set
$$F_S(X) = \prod_{(i,j) \in \odc 1,p \fdc \times \odc 1,q \fdc} X(i,j)^{S(i,j)},$$
with the convention $0^0=1$. 
\end{defi}

We note that $0 \le F_S(X) \le 1$, and that $F_S(X)>0$ if and 
only if $\supp(S) \subset \supp(X)$. 

\begin{lemm}\label{upper bound}
Let $S \in \Gamma_1$. For every $X \in \Gamma_1$, $F_S(X) \le F_S(S)$, 
with equality if and only if $X=S$. Moreover, if $\supp(S) \subset \supp(X)$, 
then $D(S||X) = \ln(F_S(S)/F_S(X))$.  
\end{lemm}

\begin{proof}
Assume that $\supp(S) \subset \supp(X)$. The definition of $F_S$ and the 
arithmetic-geometric inequality yield
$$\frac{F_S(S)}{F_S(X)} 
= \prod_{i,j} \Big(\frac{X(i,j}{S(i,j)}\Big)^{S(i,j)} 
\le \sum_{i,j} S(i,j) \Big(\frac{X(i,j}{S(i,j)}\Big) 
= \sum_{i,j}X(i,j) = 1,$$
with equality if and only if $X(i,j) = S(i,j)$ for every 
$(i,j) \in \supp(S)$. The result follows. 
\end{proof}

\begin{lemm}\label{effect of $T_R$ and $T_C$}
Let $X \in \Gamma_1$. 
\begin{itemize}
\item For every $S \in \Gamma_R$ such that $\supp(S) \subset \supp(X)$, 
one has $F_S(X) \le F_S(T_R(X))$, and the ratio $F_S(X)/F_S(T_R(X))$
does not depend on $S$. 
\item For every $S \in \Gamma_C$ such that $\supp(S) \subset \supp(X)$, 
one has $F_S(X) \le F_S(T_C(X))$, and the ratio $F_S(X)/F_S(T_C(X))$
does not depend on $S$.
\end{itemize}
\end{lemm}

\begin{proof}
Let $S \in \Gamma_R$. For every $(i,j) \in \supp(X)$, 
$X(i,j)/(T_R(X))(i,j) = R_i(X)$ so the arithmetic-geometric mean inequality 
yields
\begin{eqnarray*}
\frac{F_S(X)}{F_S(T_R(X))}
= \prod_{i,j} R_i(X)^{S_{i,j}}
= \prod_i R_i(X)^{a_i} 
\le \sum_i a_iR_i(X)
= \sum_{i,j} X(i,j)
= 1.
\end{eqnarray*}
The first statement follows. The second statement is proved in the same way. 
\end{proof} 

\begin{coro}~\label{convergence of $F_S(X_n)$}
Assume that $\Gamma(X_0)$ is not empty. Then: 
\begin{enumerate}
\item for every $S \in \Gamma(X_0)$, the sequence $(F_S(X_n))_{n \ge 0}$ 
is non-decreasing and bounded above, so it converges ; 
\item for every $(i,j)$ in the union of the supports $\supp(S)$ over all 
$S \in \Gamma(X_0)$, the sequence $(X_n(i,j))_{n \ge 1}$ is bounded away 
from $0$. 
\end{enumerate}
\end{coro}

\begin{proof}
Lemmas~\ref{upper bound} and~\ref{effect of $T_R$ and $T_C$} yield 
the first item. Given $S \in \Gamma(X_0)$ and $(i,j) \in \supp(S)$, 
we get for every $n \ge 0$, 
$X_n(i,j)^{S(i,j)} \ge F_S(X_n) \ge F_S(X_0) > 0$. 
The second item follows. 
\end{proof}

The first item of corollary~\ref{convergence of $F_S(X_n)$}
will yield the first items of theorem~\ref{slow convergence}. 
The second item of corollary~\ref{convergence of $F_S(X_n)$}
is crucial to establish the geometric rate of convergence in theorem
~\ref{fast convergence} and the convergence in theorem
~\ref{slow convergence}. 




\section{Proof of theorem~\ref{fast convergence}}

In this section, we assume that $\Gamma$ contains some matrix having 
the same support as $X_0$, and we establish the convergences with at 
least geometric rate stated in theorem~\ref{fast convergence}. 
The main tools are lemma~\ref{decreasing the length} and the 
second item of corollary~\ref{convergence of $F_S(X_n)$}.
Corollary~\ref{convergence of $F_S(X_n)$}
shows that the non-zero entries of all matrices $X_n$ are bounded below
by some positive real number $\gamma$. Therefore, the non-zero entries of all 
matrices $X_nX_n^\top$ are bounded below by $\gamma^2$. 
These matrix have the same support as $X_0X_0^\top$.

By lemma~\ref{stochastic matrix}, for every $n \ge 1$, 
$R(X_{2n+2}) = P(X_{2n})R(X_{2n})$, where $P(X_{2n})$
is a stochastic matrix given by 
$$P(X_{2n})(i,i') 
= \sum_{j=1}^q \frac{X_{2n+1}(i,j)X_{2n+1}(i',j)}{a_ib_jC_j(X_{2n+1})}.$$
These matrices have also the same support as $X_0X_0^\top$.
Moreover, by lemma~\ref{stochastic matrix} and 
corollary~\ref{nested intervals},
$$P(X_{2n})(i,i') \ge \frac{1}{\overline{a}~\overline{b}~\overline{C}(X_3)}
(X_{2n+1}X_{2n+1}^\top)(i,i'),$$
so the non-zero entries of $P(X_{2n})$ are bounded below by 
$\gamma^2/(\overline{a}~\overline{b}~\overline{C}(X_3))>0$. 

We now define a binary relation on the set $\odc 1,p \fdc$ by
$$i \rc i' \Leftrightarrow X_0X_0^\top(i,i')>0 
\Leftrightarrow \exists j \in \odc 1,q \fdc,~X_0(i,j)X_0(i',j)>0.$$ 
The matrix $X_0X_0^\top$ is symmetric with positive diagonal
(since on each line, $X_0$ has at least a positive entry), so the 
relation $\rc$ is symmetric and reflexive. Call $I_1,\ldots,I_r$ 
the connected components of the graph $G$ associated to $\rc$, and  
$d$ the maximum of their diameters. For each $k$, set 
$J_k = \{j \in \odc 1,q \fdc : \exists i \in I_k : X_0(i,j)>0\}$. 
 
\begin{lemm}\label{partition and support}
The sets $J_1,\ldots,J_k$ form a partition of $\odc 1,q \fdc$ 
and the support of $X_0$ is contained in 
$I_1 \times J_1 \cup \cdots \cup I_r \times J_r$.  
Therefore, the support of $X_0X_0^\top$ is 
contained in $I_1 \times I_1 \cup \cdots \cup I_r \times I_r$, 
so one can get a block-diagonal matrix by permuting suitably 
the lines of $X_0$. 
\end{lemm}

\begin{proof}
By assumption, the sum of the entries of $X_0$ on any row or any 
column is positive. 

Given $k \in \odc 1,r \fdc$ and $i \in I_k$, there exists 
$j \in \odc 1,q \fdc$ such that $X_0(i,j)>0$, so $J_k$ is not empty.  

Fix now $j \in [1,q]$. There exists $i \in \odc 1,p \fdc$ 
such that $X_0(i,j)>0$. Such an $i$ belongs to some connected 
component $I_k$, and $j$ belongs to the corresponding $J_k$.   
If $j$ also belongs to $J_{k'}$, then $X_0(i',j)>0$ for some 
$i' \in I_{k'}$, so $X_0X_0^\top(i,i') \ge X_0(i,j)X_0(i',j) > 0$, 
hence $i$ and $i'$ belong to the connected component of $G$, 
so $k'=k$.

The other statements follow. 
\end{proof}

\begin{lemm}
For every $n \ge 1$, set $P_{2n} = P(X_{2n})$ and 
$M_n = P_{2n+2d-2}\cdots P_{2n+2}P_{2n}$. 
Call $c$ the infimum of all positive entries of all matrices 
$P_{2n}$. Then $c>0$, and for every $n \ge 1$, 
$i \in I_k$ and $i' \in I_{k'}$. 
\begin{eqnarray*}
M_n(i,i') \ge c^d & & \text{ if } k=k'.\\
M_n(i,i') = 0 & & \text{ if } k \ne k'.
\end{eqnarray*}
\end{lemm}

\begin{proof}
The positivity of $c$ has already be proved at the beginning of 
the present section. Moreover,
$$M_n(i,i') = \sum_{1 \le i_1,\ldots,i_{d-1} \le p}P_{2n+2d-2}(i,i_1)P_{2n+2d-4}(i_1,i_2)
\cdots P_{2n+2}(i_{d-2},i_{d-1})P_{2n}(i_{d-1},i').$$

If $k \ne k'$, all these products are $0$, since no path can connect 
$i$ and $i'$ in the graph $G$.  

If $k=k'$, 
one can find a path $i = i_0,\ldots,i_\ell=i'$ in the graph $G$ with length 
$\ell \le d$. Setting $i_{\ell+1} = \ldots = i_d$ if $\ell < d$, 
we get 
$$P_{2n+2d-2}(i,i_1)P_{2n+2d-4}(i_1,i_2)
\cdots P_{2n+2}(i_{d-2},i_{d-1})P_{2n}(i_{d-1},i') \ge c^d.$$
The result follows.  
\end{proof}

Keep the notations of the last lemma. Then for every $n \ge 1$, 
$R(X_{2n+2d}) = M_nR(X_{2n})$. For each $k \in \odc 1,r \fdc$, 
lemma~\ref{decreasing the length} applied to the submatrix 
$(M_n(i,i'))_{i,i' \in I_k}$ and the vector 
$L_{I_k}(X_{2n}) = (R_i(X_{2n}))_{i \in I_k}$ yields 
$$\diam(L_{I_k}(X_{2n+2d})) \le (1-c^d) \diam(L_{I_k}(X_{2n})).$$
But lemma~\ref{decreasing the length} applied to the submatrix 
$(P(X_{2n})(i,i'))_{i,i' \in I_k}$
shows that the intervals 
$$\Big[\min_{i \in I_k} R_i(X_{2n}),\max_{i \in I_k} R_i(X_{2n})\Big]$$ 
indexed by $n \ge 1$ form a non-increasing sequence. 
Therefore, each sequence $(R_i(X_{2n}))$ tends to a limit
which does not depend on $i \in I_k$, and the speed of convergence 
it at least geometric. 

Call $\lambda_k$ this limit. By lemma~\ref{partition and support}, 
we have for every $n \ge 1$, 
$$\sum_{i \in I_k} a_iR_i(X_{2n})
= \sum_{(i,j) \in I_k \times J_k} X_{2n}(i,j) 
= \sum_{j \in J_k} X_{2n}(+,j)
= \sum_{j \in J_k} b_j$$
Passing to the limit yields
$$\lambda_k \sum_{i \in I_k} a_i = \sum_{j \in J_k} b_j,$$
whereas the assumption that $\Gamma$ contains some matrix $S$ 
having the same support as $X_0$ yields
$$\sum_{i \in I_k} a_i
= \sum_{(i,j) \in I_k \times J_k} S(i,j) 
= \sum_{j \in J_k} b_j.$$
Thus $\lambda_k=1$. 

We have proved that each sequence $(R_i(X_{2n}))_{n \ge 0}$ tends to $1$ 
with at least geometric rate. The same arguments work for 
the sequences $(C_j(X_{2n+1}))_{n \ge 0}$. Therefore, each infinite product 
$R_i(X_0)R_i(X_2)\cdots$ or $C_j(X_1)C_j(X_3)\cdots$ 
converges at an at least geometric rate. 
The convergence of the sequence $(X_n)_{n \ge 0}$
with at a least geometric rate follows. 

Moreover, call $\alpha_i$ and $\beta_j$ the inverses of the infinite products 
$R_i(X_0)R_i(X_2)\cdots$ and $C_j(X_1)C_j(X_3)\cdots$ and $X_\infty$ the limit  
of $(X_n)_{n \ge 0}$. Then $X_\infty(i,j) = \alpha_i\beta_jX_0(i,j)$, so $X_\infty$
belongs to the set $\Delta_pX_0\Delta_q$. 
As noted in the introduction, we have also $X_\infty \in \Gamma$. 
It remains to prove that $X_\infty$ is the only matrix in 
$\Gamma \cap \Delta_pX_0\Delta_q$ and the only matrix which achieves 
the least upper bound of $D(Y||X_0)$ over all $Y \in \Gamma(X_0)$.

Let $E_{X_0}$ be the vector space of all matrices in $\mc_{p,q}(\rrf)$ 
which are null on $\supp(X_0)^c$ (which can be identified canonically with 
$\rrf^{\supp(X_0)}$), and $E_{X_0}^+$ be the convex subset of all non-negative 
matrices in $E_{X_0}$. The subset $E_{X_0}^{+*}$ of all matrices in 
$E_{X_0}$ which are positive on $\supp(X_0)^c$, is open in $E_{X_0}$, 
dense in $E_{X_0}^+$ and contains $X_\infty$. 
Consider the map $f_{X_0}$ from $E_{X_0}^+$ to $\rrf$ defined by 
$$f_{X_0}(Y) = \sum_{(i,j) \in \supp(X_0)} Y(i,j) \ln \frac{Y(i,j)}{X_0(i,j)},$$
with the convention $t \ln t = 0$. This map is strictly convex 
since the map $t \mapsto t \ln t$ from $\rrf_+$ to $\rrf$ is. 
Its differential at any point $Y \in E_{X_0}^{+*}$ is given by
$$\d f_{X_0}(Y)(H) 
= \sum_{(i,j) \in \supp(X_0)} \Big( \ln \frac{Y(i,j)}{X_0(i,j)} + 1 \Big) H(i,j).$$
Now, if $Y_0$ is any matrix in $\Gamma \cap \Delta_pX_0\Delta_q$ 
(including the matrix $X_\infty$), the quantities 
$\ln (Y_0(i,j)/X_0(i,j))$ can be written $\lambda_i+\mu_j$. 
Thus for every matrix $H \in E(X_0)$ with null row-sums and column-sums, 
$\d f_{X_0}(Y_0)(H) = 0$, hence the restriction 
of $f_{X_0}$ to $\Gamma(X_0)$ has a strict global minumum at $Y_0$. 
The proof is complete.     

\paragraph{The case of positive matrices.}
The proof of the convergence at an at least geometric rate 
can be notably simplified when $X_0$ has only positive entries. 
In this case, Fienberg~\cite{Fienberg} used geometric arguments 
to prove the convergence of the iterated proportional fitting procedure 
at an at least geometric rate. We sketch another proof using the 
observation made by Fienberg that the ratios 
$$\frac{X_n(i,j)X_n(i',j')}{X_n(i,j')X_n(i',j)}$$ 
are independent of $n$, since they are preserved by the transformations 
$T_R$ and $T_C$. Call $\kappa$ the least of these positive constants. 
Using corollary~\ref{nested intervals}, one checks that the 
average of the entries of $X_n$ on each row or column remains 
bounded below by some constant $\gamma>0$. 
Thus for every location $(i,j)$ and $n \ge 1$, one can find two indexes 
$i' \in \odc 1,p \fdc$ and $j' \in \odc 1,q \fdc$ such that 
$X_n(i',j) \ge \gamma$ and $X_n(i,j') \ge \gamma$, so 
$$X_n(i,j) \ge X_n(i,j)X_n(i',j') \ge \kappa 
X_n(i,j')X_n(i',j) \ge \kappa\gamma^2.$$ 
This shows that the entries of the matrices $X_n$ remain bounded 
away from $0$, so the ratios $X_{n}(i,j)/b_j$ and $X_{n}(i,j)/a_i$ 
are bounded below by some constant $c>0$ independent of $n \ge 1$, 
$i$ and $j$. Set 
$$\rho_n = \frac{\overline{R}(X_n)}{\underline{R}(X_n)} 
\text{ if $n$ is even, }
\rho_n = \frac{\overline{C}(X_n)}{\underline{C}(X_n)} 
\text{ if $n$ is odd}.$$
For every $n \ge 1$, the equalities  
$$C_j(X_{2n+1}) = \sum_{i=1}^p \frac{X_{2n}(i,j)}{b_j} \frac{1}{R_i(X_{2n+1})}$$
and lemma~\ref{decreasing the length} yields 
$$\overline{R}(X_{2n})^{-1} 
\le (1-c)\overline{R}(X_{2n})^{-1} + c\underline{R}(X_{2n})^{-1}
\le C_j(X_{2n+1}) \le 
(1-c)\underline{R}(X_{2n})^{-1} + c\overline{R}(X_{2n})^{-1}.$$
Thus, 
\begin{eqnarray*}
\rho_{2n+1}-1 
&=& \frac{\overline{C}(X_{2n+1})-\underline{C}(X_{2n+1})}
{\underline{C}(X^{(2n+1)})} \\
&\le& \frac{(1-2c)(\underline{R}(X_{2n})^{-1} - \overline{R}(X_{2n})^{-1})}
{\overline{R}(X_{2n})^{-1}} 
= (1-2c)(\rho_{2n}-1). 
\end{eqnarray*}
We prove the inequality $\rho_{2n}-1 \le (1-2c)(\rho_{2n-1}-1)$ in the same way. 
Hence $\rho_n \to 1$ at an at least geometric rate. The result follows 
by corollary~\ref{nested intervals}. 

\section{Proof of theorem~\ref{slow convergence}}

We now assume that $\Gamma$ contains some matrix with support 
included in $\supp(X_0)$. 

\subsection{Asymptotic behavior of the sequences $(R(X_{n}))$ and $(C(X_{n}))$.}

The first item of corollary~\ref{convergence of $F_S(X_n)$}
yields the convergence of the infinite product 
$$\prod_i R_i(X_0)^{a_i} \times \prod_j C_j(X_1)^{b_j} \times 
\prod_i R_i(X_2)^{a_i} \times \prod_j C_j(X_3)^{b_j} \times \cdots.$$
Set $g(t) = t - 1 - \ln t$ for every $t>0$. Using the equalities 
$$\forall n \ge 1,~\sum_{i=1}^p a_i R_i(X_{2n}) 
= \sum_{j=1}^q b_j C_j(X_{2n-1}) = 1,$$
we derive the convergence of the series 
$$\sum_i a_ig(R_i(X_0)) + \sum_j b_jg(C_j(X_1)) 
+ \sum_i a_ig(R_i(X_2)) + \sum_j b_jg(C_j(X_3)) + \cdots.$$
But $g$ is null at $1$, positive everywhere else, and tends 
to infinity at $0+$ and at $+\infty$. By positivity of the 
$a_i$ and $b_j$, we get the convergence of all series 
$$\sum_{n \ge 0} g(R_i(X_{2n})) \text{ and } \sum_{n \ge 0} g(C_j(X_{2n+1}))$$
and the convergence of all sequences $(R_i(X_{2n}))_{n \ge 0}$ and 
to $(C_j(X_{2n+1}))_{n \ge 0}$ towards $1$. 
But $g(t) \sim (t-1)^2/2$ as $t \to 1$, so the series 
$\sum_n (R_i(X_{2n})-1)^2$ and $\sum_{n \ge 0} (C_j(X_{2n+1})-1)^2$
converge. 

We now use a quantity introduced by Bregman~\cite{Bregman} and 
called $L_1$-error by Pukelsheim \cite{Pukelsheim}. 
For every $X \in \Gamma_1$, set
\begin{eqnarray*}
e(X) 
&=& \sum_{i=1}^p\Big|\sum_{j=1}^q X(i,j)-a_i \Big| 
+ \sum_{j=1}^q\Big|\sum_{i=1}^p X(i,j)-b_j \Big| \\
&=& \sum_{i=1}^p a_i|R_i(X)-1| + \sum_{j=1}^q b_j|C_j(X)-1|.
\end{eqnarray*}
The convexity of the square function yields 
$$\frac{e(X)^2}{2} \le 
\sum_{i=1}^p a_i(R_i(X)-1)^2 + \sum_{j=1}^q b_j(C_j(X)-1)^2.$$
Thus the series $\sum_n e(X_n)^2$ converges. But the sequence 
$(e(X_n))_{n \ge 1}$ is non-increasing (the proof of this fact is 
recalled below). Therefore, for every $n \ge 1$, 
$$0 \le \frac{n}{2} e(X_n)^2 \le \sum_{n/2 \le k \le n} e(X_k)^2.$$
Convergences $ne(X_n)^2 \to 0$, $\sqrt{n}(R_i(X_{n})-1) \to 0$ and 
$\sqrt{n}(C_j(X_{n})-1) \to 0$ follow.

To check the monotonicity of $(e(X_n))_{n \ge 1}$, note that 
$T_R(X) \in \Gamma_R$ for every $X \in \Gamma_C$, so
\begin{eqnarray*}
e(T_R(X)) 
&=& \sum_{j=1}^q\Big|\sum_{i=1}^p X(i,j)~R_i(X)^{-1}-b_j \Big| \\
&=& \sum_{j=1}^q\Big|\sum_{i=1}^p X(i,j)~(R_i(X)^{-1}-1) \Big| \\
&\le& \sum_{i=1}^p \sum_{j=1}^q X(i,j)~|R_i(X)^{-1}-1| \\
&=& \sum_{i=1}^p a_iR_i(X)~|R_i(X)^{-1}-1| \\
&=& e(X).
\end{eqnarray*}
In the same way, $e(T_C(Y)) \le e(Y)$ for every $Y \in \Gamma_R$.

\subsection{Convergence and limit of $(X_n)$}\label{convergence and limit}

Since $\Gamma(X_0)$ is not empty, we can fix a matrix 
$S_0 \in \Gamma(X_0)$ whose support is maximum, like in 
theorem~\ref{criteria}, critical case, item~(a).

Let $L$ be a limit point of the sequence $(X_n)_{n \ge 0}$, 
so $L$ is the limit of some subsequence $(X_{\varphi(n)})_{n \ge 0}$. 
As noted in the introduction, $\supp(L) \subset \supp(X_0)$. 
But for every $i \in \odc 1,p \fdc$ and $j \in \odc 1,q \fdc$, 
$R_i(L) = \lim R_i(X_{\varphi(n)}) = 1$ and 
$C_j(L) = \lim C_j(X_{\varphi(n)}) = 1$. 
Hence $L \in \Gamma(X_0)$.
Corollary~\ref{convergence of $F_S(X_n)$} yields the inclusion 
$\supp(S_0) \subset  \supp(L)$ hence for every $S \in \Gamma(X_0)$, 
$\supp(S) \subset \supp(S_0) \subset \supp(L) \subset \supp(X_0)$, 
so the quantities $F_S(X_0)$ and $F_S(L)$ are positive. 

By lemma~\ref{effect of $T_R$ and $T_C$}, the ratios 
$F_S(X_{\varphi(n)})/F_S(X_0)$ do not depend on $S \in \Gamma(X_0)$, 
so by continuity of $F_S$, the ratio $F_S(L)/F_S(X_0)$ 
does not depend on $S \in \Gamma(X_0)$. 
But by lemma~\ref{upper bound}, 
$$\ln\frac{F_S(L)}{F_S(X_0)} 
= \ln\frac{F_S(S)}{F_S(X_0)} - \ln\frac{F_S(S)}{F_S(L)} 
= D(S||X_0)-D(S||L),$$
and $D(S||L) \ge 0$ with equality if and only if $S=L$. 
Therefore, $L$ is the only element achieving 
the greatest lower bound of $D(S||X_0)$ over all $S \in \Gamma(X_0)$.
$$L = \argmin_{S \in \Gamma(X_0)} D(S||X_0).$$
We have proved the unicity of the limit point of the sequence 
$(X_n)_{n \ge 0}$. By compactness of $\Gamma(X_0)$, the convergence follows. 

\begin{rema}
Actually, one has $\supp(S_0) = \supp(L)$. 
Indeed, theorem~\ref{slow convergence} shows that for every 
$(i,j) \in \supp(X_0) \setminus \supp(S_0)$, $X_n(i,j) \to 0$ 
as $n \to +\infty$. This fact could be retrived by using 
the same arguments as in the proof of theorem~\ref{criteria}
to show that on the set $\Gamma(X_0)$, the linear form 
$X \mapsto X(i_0,j_0)$ coincides with some linear combination of the 
affine forms $X \mapsto R_i(X)-1$, $i \in \odc 1,p \fdc$, and 
$X \mapsto C_j(X)-1$, $j \in \odc 1,q \fdc$. 
\end{rema}

\section{Proof of theorems~\ref{divergence},
~\ref{determining the partitions} and 
~\ref{modification of the initial matrix}} 

We recall that neither proof below uses the assumption that 
$\Gamma(X_0)$ is empty.

\subsection{Proof of theorem~\ref{divergence}} 

\paragraph{Convergence of the sequences $(R(X_{2n}))$ and
$(C(X_{2n+1}))$.}
By lemma~\ref{stochastic matrix} and corollary~\ref{nested intervals}, 
we have for every $n \ge 1$, $R(X_{2n+2}) = P(X_{2n})R(X_{2n})$, where $P(X_{2n})$
is a stochastic matrix 
such that for every $i$ and $k$ in $\odc 1,p \fdc$, 
$$P(X_{2n})(i,i) 
\ge \frac{\underline{a}}{\overline{b}~\overline{C}(X_{2n+1})q}
\ge \frac{\underline{a}~\underline{R}(X_2)}{\overline{b}q}$$
and
$$P(X_{2n})(k,i) \le (\overline{a}/\underline{a}) P(X_{2n})(i,k).$$
The sequence $(P(X_{2n}))_{n \ge 0}$ satisfies the assumption of 
corollary~\ref{finite variation} and 
theorem~\ref{infinite product of stochastic matrices}, so any one of 
these two results ensures the convergence of the sequence 
$(R(X_{2n}))_{n \ge 0}$. By corollary~\ref{nested intervals}, 
the entries of these vectors stay in the interval 
$[\underline{R}(X_2),\overline{R}(X_2)]$, so the limit of each 
entry is positive. 
The same arguments show that the sequence $(C(X_{2n+1}))_{n \ge 0}$ also 
converges to some vector with positive entries. 

\paragraph{Relations between the components of the limits, and block structure.}
Denote by $\lambda_1<\ldots<\lambda_r$ the different values of the limits 
of the sequences $(R_i(X_{2n}))_{n \ge 0}$, and by $\mu_1>\ldots>\mu_s$ 
the different values of the limits of the sequences $(C_j(X_{2n+1}))_{n \ge 0}$.
The values of these limits will be precised later.
Consider the sets
$$I_k = \{i \in \odc 1,p \fdc : \lim R_i(X_{2n}) = \lambda_k\}
\text{ for } k \in \odc 1,r \fdc,$$ 
$$J_l = \{j \in \odc 1,q \fdc : \lim C_j(X_{2n+1}) = \mu_l\}
\text{ for } l \in \odc 1,s \fdc.$$

When $(i,j) \in I_k \times J_l$, the sequence 
$(R_i(X_{2n})C_j(X_{2n+1}))_{n \ge 0}$ converges to $\lambda_k\mu_l$. 
If $\lambda_k\mu_l>1$, this entails the convergence to $0$ of 
the sequence $(X_n(i,j))_{n \ge 0}$ with a geometric rate; 
and if $\lambda_k\mu_l<1$, 
this entails the nullity of all $X_n(i,j)$ (otherwise the sequence 
$(X_n(i,j))_{n \ge 0}$ would go to infinity). 
But for all $n \ge 1$, $R_i(X_{2n})=1$ and $C_j(X_{2n+1}) = 1$, so 
at least one entry of the matrices $X_n$ on each line or column does 
not converge to $0$. This forces the equalities $s=r$ and 
$\mu_k = \lambda_k^{-1}$ for every $k \in \odc 1,r \fdc$.

\paragraph{Convergence of the sequences $(X_{2n})$ and
$(X_{2n+1})$.}
Let $L$ be any limit point of the sequence $(X_{2n})_{n \ge 0}$, 
so $L$ is the limit of some subsequence $(X_{2\varphi(n)})_{n \ge 0}$. 
By definition of $a'$, $R_i(L) = \lim R_i(X_{2\varphi(n)}) = a'_i/a_i$ 
for every $i \in \odc 1,p \fdc$. Moreover, $\supp(L) \subset \supp(X_0)$, 
so $L$ belongs to $\Gamma(a',b,X_0)$. 

Like in subsection~\ref{convergence and limit}, we check that the quantity 
$$D(S||X_0)-D(S||L) = \ln (F_S(L)/F_S(X_0))$$ 
does not depend on $S \in \Gamma(a',b,X_0)$, so $L$ is the unique matrix 
achieving the greatest lower bound of $D(S||X_0)$ over all 
$S \in \Gamma(a',b,X_0)$. The convergence of $(X_{2n})_{n \ge 0}$ follows 
by compactness of $\Gamma(a',b,X_0)$. 

By lemma~\ref{effect of $T_R$ and $T_C$}, the ratios 
$F_S(X_{2\varphi(n)})/F_S(X_0)$ do not depend on $S \in \Gamma(X_0)$, 
so by continuity of $F_S$, the ratio $F_S(L)/F_S(X_0)$ 
The same arguments show that the sequence $(X_{2n+1})_{n \ge 0}$ 
converges to the unique matrix achieving the greatest lower bound of 
$D(S||X_0)$ over all $S \in \Gamma(a,b',X_0)$.

\paragraph{Formula for $\lambda_k$.}
We know that the sequence $(X_n(i,j))_{n \ge 0}$ converges to $0$ 
whenever $i \in I_k$ and $j \in J_l$ with $k \ne l$. Thus the 
support of $L = \lim X_{2n}$ is contained in 
$I_1 \times J_1 \cup \ldots \cup I_r \times J_r$. 
But $L$ belongs to $\Gamma(a',b,X_0)$, so for every 
$k \in \odc 1,r \fdc$ 
$$\lambda_k a(I_k) = \sum_{i \in I_k} a'_i 
= \sum_{(i,j) \in I_k \times J_k} L(i,j) 
= \sum_{j \in J_k} b_j = b(J_k).$$

\paragraph{Properties of matrices in $\Gamma(a',b,X_0)$ and $\Gamma(a',b,X_0)$.}
Let $S \in \Gamma(a,b',X_0)$. 

Let $k \in \odc 1,r-1 \fdc$, $A_k = I_1 \cup \cdots \cup I_k$ and 
$B_k = J_{k+1} \cup \cdots \cup I_r$. We already know that 
$X_0$ is null on $A_k \times B_k$, so $S$ is also null on this set. 
Moreover, for every $l \in \odc 1,r \fdc$, 
$$a(I_l) = \lambda_l^{-1} b(J_l) = \sum_{j \in J_l} \lambda_l^{-1}b_j 
= \sum_{j \in J_l} b'_j = b'(J_l).$$
Summation over all $l \in \odc 1,k \fdc$ yields $a(A_k) = b'(B_k^c)$. 
Hence by theorem~\ref{criteria} (critical case), 
$S$ is null on the set $A_k^c \times B_k^c 
= (I_{k+1} \cup \cdots \cup I_r) \times (J_1 \cup \cdots \cup J_k)$. 

This shows that the support of $S$ is included in 
$I_1 \times J_1 \cup \cdots \cup I_r \times J_r$. 
This block structure and the equalities $a'_i/a_i = b_j/b'_j = \lambda_k$ 
whenever $(i,j) \in I_k \times J_k$ yield the equality $D_1S=SD_2$. 
This matrix has the same support as $S$.
Moreover, its $i$-th row is $a'_i/a_i$ times the $i$-th row of $S$, 
so its $i$-th row-sum is $a'_i/a_i \times a_i = a'_i$; 
in the same way its $j$-th column is $b_j/b'_j$ times the $j$-th column of $S$, 
so its $j$-th column sum is $b_j/b'_j \times b'_j = b_j$. 
As symmetric conclusions hold for every matrix in $\Gamma(a',b,X_0)$, 
the proof is complete. 

\subsection{Proof of theorem~\ref{determining the partitions}} 

Let $k \in \odc 1,r \fdc$, 
$P = \odc 1,p \fdc \setminus (I_1 \cup \ldots \cup I_{k-1})$, 
$Q = \odc 1,q \fdc \setminus (J_1 \cup \ldots \cup J_{k-1})$.
Fix $S \in \Gamma(a',b,X_0)$
(we know by therorem~\ref{divergence} that this set is not empty). 

If $k=r$, then $P = I_r$ and $Q = J_r$. As $a'_i = a_ib(J_r)/a(I_r)$
for every $i \in I_r$,  we have $a'(I_r)=b(J_r)$ and 
$a'_i/a'(I_r)= a_ib(J_r)/a(I_r)$ for every $i \in I_r$. Therefore, 
the matrix $(S(i,j)/a'(I_r))$ is a solution of the restricted 
problem associated to the marginals 
$a(\cdot|P) = (a_i/a(I_r))_{i \in P}$, $b(\cdot|Q) = (b_j/b(Q))_{j \in Q}$
and to the initial condition
$(X_0(i,j))_{(i,j) \in P \times Q}$. 

If $k<r$, then $P = I_k \cup \ldots \cup I_r$ and 
$Q = J_k \cup \ldots \cup J_r$. 
Let $A_k = I_k$ and $B_k = J_{k+1} \cup \ldots \cup J_r$.
By theorem~\ref{divergence}, the matrix $X_0$ is null on product 
$A_k \times B_k$. Moreover, the inequalities 
$\lambda_1 < \ldots < \lambda_r$ and $a(I_l)>0$ for every 
$l \in \odc 1,r \fdc$ yield
$$b(Q) = \sum_{l=k}^r b(J_l) = \sum_{l=k}^r \lambda_l a(I_l) > 
\lambda_k \sum_{l=k}^r a(I_l) = \lambda_k a(P),$$
so
$$\frac{a(A_k|P)}{b(Q \setminus B_k|Q)} 
= \frac{a(I_k)/a(P)}{b(J_k)/b(Q)}
= \lambda_k^{-1} \frac{b(Q)}{a(P)} > 1.$$ 
Hence $A_k \times B_k$ is a cause of incompatibility
of the restricted problem associated to the marginals 
$a(\cdot|P) = (a_i/a(P))_{i \in P}$, $b(\cdot|Q) = (b_j/b(Q))_{j \in Q}$
and to the initial condition
$(X_0(i,j))_{(i,j) \in P \times Q}$. 

Now, assume that $X_0$ is null on some subset $A \times B$ of $P \times Q$. 
Then $S$ is also null on $A \times B$, 
so for every $l \in \odc k,r \fdc$, 
$$\lambda_k a(A \cap I_l) \le \lambda_l a(A \cap I_l) = a'(A \cap I_l) 
= S((A \cap I_l) \times ((Q \setminus B) \cap J_l)) 
\le b((Q \setminus B) \cap J_l).$$
Summing this inequalities over all $l \in \odc k,r \fdc$ yields
$\lambda_k a(A) \le b(Q \setminus B)$, so 
$$\frac{a(A)}{b(Q \setminus B)} \le \lambda_k^{-1} = \frac{a(I_k)}{b(I_k)}
= \frac{a(A_k)}{b(Q \setminus B_k)}.$$ 
Moreover, if equality holds in the last inequality, then 
for every $l \in \odc k,r \fdc$, 
$$\lambda_k a(A \cap I_l) = \lambda_l a(A \cap I_l) = 
b((Q \setminus B) \cap J_l).$$ 
This yields $A \cap I_l = (Q \setminus B) \cap J_l = \emptyset$ 
for every $l \in \odc k+1,r \fdc$, thus $A \subset I_k = A_k$ and 
$Q \setminus B \subset I_k$, namely $B \supset B_k$.
The proof is complete. 

\subsection{Proof of theorem~\ref{modification of the initial matrix}}

The proof relies the next two lemmas, from Pretzel, relying on notion 
of diagonal equivalence. We provide proofs to keep the paper self-contained. 
The first one is different from Pretzel's original proof. 
Recall that two matrices $X$ and $Y$ in $\mc{p,q}(\rrf_+)$
are said to be diagonally equivalent if there exists $D' \in \Delta_p$ and 
$D'' \in \Delta_q$ such that $Y = D'XD''$. In particular, $X$ and $Y$ 
must have the same support to be diagonally equivalent. 

\begin{lemm}\label{diagonal equivalence and marginals}
{\bf (Property 1 of \cite{Pretzel})}
Let $X$ and $Y$ be in $\mc{p,q}(\rrf_+)$. If $X$ and $Y$ are diagonally 
equivalent and have the same marginals then $X=Y$.  
\end{lemm}
 
\begin{proof}
By assumption, $Y = D'XD''$ for some $D' \in \Delta_p$ and 
$D'' \in \Delta_q$. Call $\alpha_1,\ldots,\alpha_p$ and 
$\beta_1,\ldots,\beta_q$ the diagonal entries of $D'$ and $D''$. 
For every $(i,j) \in \odc 1,p \fdc \times \in \odc 1,q \fdc$, 
$Y(i,j) = \alpha_i\beta_jX(i,j)$, so 
$$D(Y||X) = \sum_{i,j} Y(i,j) (\ln\alpha_i+\ln\beta_j) 
=  \sum_i Y(i,+)\ln\alpha_i + \sum_j Y(+,j)\ln\beta_j.$$
In the same way, 
$$D(X||Y) = \sum_{i,j} X(i,j) (\ln(\alpha_i^{-1})+\ln(\beta_j^{-1})) 
= - \sum_i X(i,+)\ln\alpha_i - \sum_j X(+,j)\ln\beta_j.$$
Since $X$ and $Y$ have the same marginals, the non-negative 
quantities $D(Y|X)$ and $D(X|Y)$ are opposite, so they are null.
Hence $X=Y$.
\end{proof}

\begin{lemm}\label{diagonal equivalence and limit}
{\bf (Lemma 2 of \cite{Pretzel})}
Let $X$ and $Y$ be in $\mc{p,q}(\rrf_+)$. If $Y$ has the same support 
as $X$ and is the limit of a some sequence $(Y_n)_{n \ge 0}$ of 
matrices which are diagonally equivalent to $X$, then $X$ and $Y$ 
are diagonally equivalent. 
\end{lemm}

\begin{proof}
For every $n \ge 0$, one can find some positive real numbers 
$\alpha_n(1),\ldots,\alpha_n(p)$ and $\beta_n(1),\ldots,\beta_n(q)$
such that $Y_n(i,j) = \alpha_n(i)\beta_n(j)X(i,j)$ for every 
$(i,j) \in \odc 1,p \fdc \times \in \odc 1,q \fdc$. 
By assumption, the sequence $(\alpha_n(i)\beta_n(j))_{n \ge 0}$ 
converges to a positive number whenever $(i,j) \in \supp(X)$. 

In a way similar to the beginning of the proof of theorem~\ref{divergence}, 
we define a non-oriented graph $G$ on $\odc 1,p \fdc$ as follows:
$(i,i')$ is an edge if and only if there exists some $j \in \odc 1,q \fdc$
such that $X(i,j)X(i',j)>0$. Then the sequence 
$(\alpha_n(i)/\alpha_n(i'))_{n \ge 0}$ converges whenever $i$ and $i'$ 
belong to a same connected component of $G$. 

Call $I_1,\ldots,I_r$ the connected components of the graph $G$. 
For each $k \in \odc 1,r \fdc$ choose $i_k \in I_k$ and set 
$J_k = \{j \in \odc 1,q \fdc : \exists i \in I_k : X(i,j)>0\}$. 
Then the sets $J_1,\ldots,J_k$ form a partition of $\odc 1,q \fdc$ 
and the support of $X$ is contained in 
$I_1 \times J_1 \cup \cdots \cup I_r \times J_r$.  

For every $n \ge 0$, set $\alpha'_n(i) = \alpha_n(i)/\alpha_n(i_k)$ 
whenever $i \in I_k$ and $\beta'_n(j) = \beta_n(j)\alpha_n(i_k)$ 
whenever $j \in J_k$. Then $Y_n(i,j) = \alpha'_n(i)\beta'_n(j)X(i,j)$ 
for every $(i,j) \in \odc 1,p \fdc \times \in \odc 1,q \fdc$. 
Since all sequences $(\alpha'_n(i))_{n \ge 0}$ and $(\beta'_n(j))_{n \ge 0}$ 
converge to a positive limit, we deduce that $X$ and $Y$ are diagonally 
equivalent.
\end{proof}

We now prove theorem~\ref{modification of the initial matrix}.

Set $L_{\rm even} = \lim_n X_{2n}$ and $L_{\rm odd} = \lim_n X_{2n+1}$. 
Letting $n$ go to infinity in the equality $X_{2n+1} = T_R(X_{2n})$ 
yields $L_{\rm odd} = D_1^{-1}L_{\rm even}(i,j)$, where 
$D_1 = \diag(a'_1/a_1,\ldots,a'_p/a_p)$.   
We deduce that $L_{\rm even}$ and $L_{\rm odd}$ 
have the same support $\Sigma$. 

By theorem~\ref{divergence}, $L_{\rm even}$ is the only matrix achieving 
the minimum of $D(Y||X_0)$ over all $Y \in \Gamma(a',b,X_0)$. Thus, 
by theorem~\ref{slow convergence}, $L_{\rm even}$ is also the limit 
of the sequence provided by the IPFP performed with the marginals 
$a'$ and $b$ and the initial matrix $X_0$, so $\Sigma$ is 
the maximum of the supports of all matrices in $\Gamma(a',b,X_0)$. 

By construction, for every $n \ge 0$, one 
can find $D'_n \in \Delta_p$ and $D''_n \in \Delta_q$ such that 
$X_{2n} = D'_nX_0D''_n$. Since $D'_n$ and $D''_n$ are diagonal, 
$(D'_nX'_0D''_n)(i,j) = X_{2n}(i,j)$ whenever $(i,j) \in \Sigma$ 
and $(D'_nX'_0D''_n)(i,j) = 0$ whenever $(i,j) \in \Sigma^c$. 
Hence $\lim_n D'_nX'_0D''_n = L_{\rm even}$ since $\Sigma = \supp(L_{\rm even})$. 
But $X'_0$ and $L_{\rm even}$ have the same support, they
are diagonally equivalent by 
lemma~\ref{diagonal equivalence and limit}. 

Call $(X'_n)_{n \ge 0}$ (respectively $(X''_n)_{n \ge 1}$) the sequence 
provided by the IPFP performed on the marginals $a,b$ (respectively $a',b$)
and the initial matrix $X'_0$. 
Equivalently, one could also start from $X'_0(+,+)^{-1}X'_0$ to 
have an initial matrix in $\Gamma_1$.  

Since $L_{\rm even} \in \Gamma(a',b)$ and $\supp(L_{\rm even})=\supp(X'_0)$, 
theorem~\ref{fast convergence} applies, so the limit 
$L'' = \lim_n X''_n$ exists, $L''$ belongs to $\Gamma(a',b)$ and 
$L''$ is diagonally equivalent to $X'_0$ and therefore to $L_{\rm even}$. 
By lemma~\ref{diagonal equivalence and marginals}, we get $L''=L_{\rm even}$. 
Since all matrices $X'_n$ and $X''_n$ have the same support $\Sigma$, 
contained in $I_1 \times J_1 \cup \cdots \cup I_1 \times J_1$
by theorem~\ref{divergence}, a recursion shows that for every $n \ge 0$, 
$X'_{2n} = X''_{2n}$ and $X'_{2n+1} = D_1^{-1}X''_{2n+1}$. 
Hence $\lim_n X'_{2n} = L'' = L_{\rm even}$.

A similar proof works for $L_{\rm odd}$ and the set $\Gamma(a,b',X_0)$.

\paragraph{Acknowledgement.} 
We thank A.~Coquio, D.~Piau, F.~Pukelsheim and the referee for their 
careful reading and useful remarks. We also thank G.~Geenens who pointed 
our attention on a wrong statement in theorem~\ref{criteria}, 
critical case, item c.

\end{document}